\documentclass[reqno,12pt,a4paper]{amsart}
\usepackage{amsmath,amssymb,amsthm,graphicx,mathrsfs,url,bbm}
\usepackage{a4wide}

\usepackage[english]{babel}

\usepackage{csquotes}
\MakeOuterQuote{"} \usepackage{enumitem}
\usepackage{mathtools, stackrel}
\usepackage[usenames,dvipsnames]{xcolor}
\usepackage[colorlinks=true,linkcolor=Blue,citecolor=Green]{hyperref}
\hypersetup{pdfstartview=XYZ}

\usepackage{tikz}
\usetikzlibrary{fit,shapes.geometric,backgrounds}
\usepackage{tikz-cd}
\usepackage{dsfont}

\theoremstyle{plain}
\newtheorem{thm}{Theorem}
\newtheorem{lem}{Lemma}[section]
\newtheorem{prop}[lem]{Proposition}

\theoremstyle{definition}
\newtheorem{definition}[lem]{Definition}
\newtheorem{ex}[lem]{Example}

\theoremstyle{remark}
\newtheorem{rem}{Remark}[section]

\numberwithin{equation}{section}

\usepackage[normalem]{ulem}

\newcommand{\C}{\mathbb{C}}
\newcommand{\R}{\mathbb{R}}

\newcommand{\Z}{\mathbb{Z}}

\newcommand{\N}{\mathbb{N}}

\newcommand*\Laplace{\mathop{}\!\mathbin\bigtriangleup}

\def\Ddots{\mathinner{\mkern1mu\raise\p@
    \vbox{\kern7\p@\hbox{.}}\mkern2mu
    \raise4\p@\hbox{.}\mkern2mu\raise7\p@\hbox{.}\mkern1mu}}
\makeatother

\renewcommand{\epsilon}{\vararepsilon}

\newcommand{\bdm}{\begin{displaymath}}
  \newcommand{\edm}{\end{displaymath}}
\newcommand{\bq}{\begin{equation}}
  \newcommand{\eq}{\end{equation}}
\newcommand{\bqn}{\begin{equation*}}
  \newcommand{\eqn}{\end{equation*}}

\newcommand{\eop}[1]{\ensuremath{\vec{e}^{\,\mathrm{op}}}}

\title[Patterson-Sullivan distributions of finite regular graphs]{Patterson-Sullivan distributions of finite regular graphs}

\author[C.~Arends]{Christian Arends}
\address{Department of Mathematics, Aarhus University, Ny Munkegade 118, 8000 Aarhus C, Denmark}
\email{arends@math.au.dk}

\author[G.~Palmirotta]{Guendalina Palmirotta}
\address{Institute of Mathematics, Universit\"at Paderborn, Warburgerstr. 100, 33098 Paderborn, Germany} \email{gpalmi@math.uni-paderborn.de}

 \date{\today}

\begin{document}

 \begin{abstract}
   On finite regular graphs, we construct Patterson–Sullivan distributions associated with eigenfunctions of the discrete Laplace operator via their boundary values on the phase space.
   These distributions are closely related to Wigner distributions defined via a pseudo-differential calculus on graphs, which appear naturally in the study of quantum chaos. Using a pairing formula, we prove that Patterson–Sullivan distributions are also related to invariant Ruelle distributions arising from the transfer operator of the geodesic flow on the shift space.
   Both relationships provide discrete analogues of results for compact hyperbolic surfaces obtained by Anantharaman–Zelditch and by Guillarmou–Hilgert–Weich.\\

 \end{abstract}

 \maketitle

 \section{Introduction}
 Consider a finite $(q+1)$-regular graph $\mathfrak{G}_\Gamma$ without dead ends, for $q \ge 2$. It is the quotient
 $\mathfrak{G}_\Gamma = \Gamma \backslash \mathfrak{G}$
 of a homogeneous tree $\mathfrak{G} = (\mathfrak{X}, \mathfrak{E})$, with vertex set $\mathfrak{X}$ and directed edge set $\mathfrak{E}$, by a subgroup $\Gamma \le \mathrm{Aut}(\mathfrak{X})$ of automorphisms whose action on $\mathfrak{G}$ is cocompact due to the finiteness of $\mathfrak{G}_\Gamma$. The corresponding phase space of $\mathfrak{G}_\Gamma$ is the quotient $S\mathfrak{X}_\Gamma= \Gamma \backslash (\mathfrak{X} \times \Omega)$, where $\Omega$ is the boundary at infinity of $\mathfrak{G}$. It consists of semi-geodesic trajectories $[x, \omega)$ starting at a vertex $x\in \mathfrak{X}$ and heading towards $\omega \in \Omega$, with dynamics given by the non-backtracking shift, and serves as a discrete analogue of the unit sphere bundle in the Archimedean setting.

 Let $\Laplace$ denote the vertex Laplace operator on $\mathfrak{G}$. For a spectral parameter $s \in \mathbb{C}$, we consider the eigenvalue problem
 $$
 \Laplace \phi = \chi(s) \phi,
 $$
 where $\chi(s) \in \mathbb{C} \setminus \{\pm 1\}$ is the eigenvalue and $\phi \in \mathcal{E}_{\chi(s)}(\Laplace ; \mathrm{Maps}(\mathfrak{X}, \mathbb{C}))^\Gamma$ is a Laplace $\Gamma$-eigenfunction in the associated eigenspace.

 Given $s, s' \in \C$ and Laplace $\Gamma$-eigenfunctions
 $\phi \in \mathcal{E}_{\chi(s)}(\Laplace ; \mathrm{Maps}(\mathfrak{X}, \mathbb{C}))^\Gamma,\, \phi' \in \mathcal{E}_{\chi(s')}(\Laplace;$ $ \mathrm{Maps}(\mathfrak{X}, \mathbb{C}))^\Gamma$, one can associate a special type of phase space distribution, the (off-diagonal) \emph{Patterson–Sullivan distribution} $\mathrm{PS}^\Gamma_{\phi,\phi'} \in \mathcal{D}'(S\mathfrak{X}_\Gamma).$ Note that when $\phi=\bar \phi'$ and $s= -\bar s'$, 
 we refer to the corresponding distributions as ``diagonal'' Patterson-Sullivan distributions.
 The terminology originates from the Archimedean setting, where Anantharaman and Zelditch first introduced these distributions on compact hyperbolic surfaces \cite{AZ07}. They were later extended to compact locally symmetric spaces of non-positive curvature by Hansen-Hilgert-Schröder \cite{HHS12} and, more recently, to convex cocompact hyperbolic surfaces in \cite{DP24}.

 Mirroring the Archimedean setting \cite{AZ07, AZ12, HHS12, GHWb, DP24}, these distributions exhibit several notable properties and relationships reflecting the underlying geometry and dynamics, which we outline in the following.

 \subsection*{From the classical to the modern description (Theorem~\ref{thm:PS_modern})}
 Following their original definition in the Archimedean setting \cite{AZ07, AZ12, HHS12}, the \emph{classical Patterson–Sullivan distributions} are defined in terms of the \emph{boundary values} $\mu_{s,\phi}, \mu_{s',\phi'} \in \mathcal{D}'(\Omega)$ of the Laplace eigenfunctions $\phi,\phi'$, which are distributions defined on the boundary $\Omega$. These boundary distributions arise from the inverse of the Poisson transform, which yields an isomorphism between generalized functions on the boundary $\Omega$ and eigenfunctions of the vertex Laplacian (see Theorem~\ref{thm:Poissontransform_delta}),
 following \cite{BHW22}.
 By identifying distributions on $S\mathfrak{X}_\Gamma$ with $\Gamma$-invariant distributions on $S\mathfrak{X}$, the classical Patterson–Sullivan distributions are defined as the $\Gamma$-average of $\mathcal{R}'_{s,s'}\Big(\mu_{s,\phi} \otimes \overline{\mu_{s',\phi'}}\Big),$
 where
 $\mathcal{R}'_{s,s'}$ is the distributional weighted Radon transform, and $\overline{\phantom{a}}$ denotes complex conjugation (see Subsection~\ref{sect:Radontransform}).

 Relying on the quantum-classical correspondence, one obtains a \emph{dynamical (or modern) formulation} of these phase space distributions, in the spirit of \cite{GHWb} for the Archimedean case. More precisely, for a non-exceptional spectral parameter $s\in \C$, i.e., $\chi(s) \neq \pm 1$, the quantum-classical correspondence, established in \cite{AFH23} and adapted to our setting, asserts that the eigenspaces $\mathcal{E}_{\chi(s)}(\Delta_\Gamma; \mathrm{Maps}(\mathfrak{X}_\Gamma, \C))$ of the Laplacian $\Laplace_\Gamma$ acting on functions on $\mathfrak{X}_\Gamma$ (the \emph{quantum} side) are isomorphic to the eigenspaces $\mathcal{E}_s(\mathcal{L}'_{\Gamma, \pm }; \mathcal{D}'(\mathfrak{P}^\pm_\Gamma))$ of a dual transfer operator $\mathcal{L}'_{\Gamma, \pm }$ acting on functions on one-sided infinite non-backtracking paths $\mathfrak{P}^\pm_\Gamma$ (the \emph{classical} side), see Section~\ref{sect:PS_dynamical}.
 This correspondence allows us to associate a resonant state $u_{+,\phi}\in \mathcal{D}'(\mathfrak{P}^+_\Gamma)$ and a co-resonant state $u_{-,\phi}\in \mathcal{D}'(\mathfrak{P}^-_\Gamma)$, which are non-zero elements of $\mathcal{E}_s(\mathcal{L}'_{\Gamma, \pm }; \mathcal{D}'(\mathfrak{P}^\pm_\Gamma))$, to a non-exceptional eigenvalue (a \emph{resonance}) of $\mathcal{L}'_{\Gamma, \pm}$.
 By identifying the phase space $S\mathfrak{X}_\Gamma$ with the space $\mathfrak{P}^{\pm}_\Gamma$ of infinite chains on the graph $\mathfrak{G}_\Gamma$,
 the Patterson–Sullivan distribution can then be expressed as a distributional tensor product of these resonant and co-resonant states:
 $$\mathrm{PS}_{\phi,\phi'}^\Gamma= u_{+,\phi}\; \otimes \; u_{-,\overline{\phi'}}\in \mathcal{D}'(\mathfrak{P}_\Gamma).$$
 Moreover, by evaluating the tensor product at the characteristic function $\mathbbm{1}_{\mathfrak{P}_\Gamma}$, one obtains the \emph{geodesic pairing formula} from \cite{AFH23Pairing} (see Remark~\ref{rem:relation_pairing_formula}).

 \subsection*{Relation to invariant Ruelle distributions (Theorem~\ref{thm:relation_PSTs})}
 In connection with the pairing formula, the dynamical Patterson–Sullivan distributions introduced above can be related to the \emph{invariant Ruelle distributions}: For a resonance $s \in \mathbb{C}$ of multiplicity $m \in \mathbb{N}$ such that $q^{\frac{1}{2} + is} \notin \{\pm 1, \pm q\}$ and for which no Jordan block occurs, we have:
 $$
 \mathcal{T}_s = \frac{q^{1 + 2 is} - 1}{q^{1 + 2 is} - q}\sum_{
   \ell = 1
 }^m \mathrm{P S}_{\phi_{\ell} , \phi'_{\ell}}^\Gamma.
 $$
 In the Archimedean setting, these distributions $\mathcal{T}_s$ were introduced by Guillarmou, Hilgert, and Weich \cite{GHWb}. They are $\mathcal{L}'_\Gamma$-invariant elements of $\mathcal{D}'(\mathfrak{P}_\Gamma)$ associated with resonances $s \in \C$ and are defined as the trace of a finite-rank operator (see Section~\ref{sect:InvariantRuelleDist}). They are of particular interest, as their structure and properties remain largely unexplored, even in the Archimedean case. Numerical investigations on Schottky surfaces, paradigmatic models of hyperbolic dynamics on non-compact manifolds, provide first initial evidence that these distributions give rise to computable spectral invariants and are closely linked to fine spectral and dynamical features of the underlying flow; see \cite{SchutteWeich23}. These observations further motivate us to introduce them and study their connection with Patterson–Sullivan distributions in the graph setting.

 \subsection*{Relation to Wigner distributions (Theorem~\ref{thm:relation_PSW})}
 There is a second kind of phase distributions that can be associated with $s,s' \in \C$ and eigenfunctions $\phi \in \mathcal{E}_{\chi(s)}(\Laplace_\Gamma ; \mathrm{Maps}(\mathfrak{X}_\Gamma, \mathbb{C})),$ $ \phi' \in \mathcal{E}_{\chi(s')}(\Laplace_\Gamma ; \mathrm{Maps}(\mathfrak{X}_\Gamma, \mathbb{C}))$, namely the \emph{Wigner distributions} $W_{\phi,\phi'} \in \mathcal{D}'(S\mathfrak{X}_\Gamma)$.
 By adapting the pseudo-differential calculus of \cite{LeMas14, ALe15} to our setting, these distributions, also known as microlocal or quantum lifts arising from quantum chaos, are defined via pseudo-differential operators on $\mathfrak{X}_\Gamma$ (see Section~\ref{sect:WignerDist}).

 It is a natural question how these two phase space families given by the Wigner and Patterson-Sullivan distributions are \emph{related}.
 In the Archimedean compact setting, they are connected \emph{asymptotically}: along sequences of eigenfunctions with unbounded eigenvalues on the critical line $\frac{1}{2} +i \R$.
 In contrast, for finite graphs, the Laplace spectrum is finite and \emph{no such unbounded sequences of eigenvalues exist}. Thus, in contrast to the Archimedean case, we do not consider asymptotics but prove exact relations between Wigner and Patterson-Sullivan distributions for each pair of \emph{fixed} resonances on finite graphs. For this we define, for each $n \in \mathbb{N}_0$, the cutoff set (see Figure~\ref{fig:Sn})
 $$S_n \coloneqq \left\{(x, \omega , \omega ') \in \mathfrak{X} \times \Omega \times \Omega \mid d(x,]\omega , \omega '[) \leq n \right\},$$
 where $]\omega , \omega '[$ denotes the geodesic between boundary points $\omega, \omega' \in \Omega$ and $d(x,]\omega , \omega '[)$ is the graph distance from $x \in \mathfrak{X}$ to this geodesic. 
 Studying the behaviour of Wigner distributions on these sets and their complements---more precisely, by decomposing them into off-diagonal and near-diagonal parts---one can express them as a weighted combination of Patterson–Sullivan distributions. 
 This idea is made precise in Theorem~\ref{thm:relation_PSW}.\\

 Hence, our Theorems~\ref{thm:PS_modern}--\ref{thm:relation_PSW} develop a theory of Patterson–Sullivan distributions for finite graphs, thereby addressing part of the problem stated in \cite[Problem 6.33]{Hilgert23} and discussed in \cite[Outlook]{AFH23Pairing}.

 Let us mention that the relation from Theorem~\ref{thm:relation_PSW} is particularly useful in the context of quantum ergodicity, where a central problem is to understand the weak$^*$-limits of Wigner distributions.
 Remarkable progress in this direction has already been achieved in the graph setting, ranging from large and random graphs to their higher-dimensional generalization given by Bruhat–Tits buildings \cite{Peterson23_QE}. We refer to the survey of Anantharaman and Sabri \cite{AnantharamanSabri19_survey} for a comprehensive overview of these developments.

 Finally, let us mention that building on a recent work with Arends, Peterson, and Weich \cite{ArendsPetersonWeich26}, it would be interesting to extend our results to geometrically finite graphs, the discrete analogues of geometrically finite hyperbolic surfaces (including funnels and cusps).

 \subsection*{Structure of the paper}
 In Section~\ref{sect:harmonic_analysis_graphs}, we introduce our geometric setting for finite regular graphs by presenting the harmonic analysis tools needed to define (vertex) Poisson transforms and $\chi(s)$-boundary values with moderate growth, along with their regularity properties.
 Section~\ref{sect:PS_Dist} discusses Patterson–Sullivan distributions, presenting both the classical Radon-transform formulation and a dynamical description via resonant and co-resonant states using the graph quantum–classical correspondence from \cite{AFH23}.
 The last two sections, Sections~\ref{sect:InvariantRuelleDist} and~\ref{sect:WignerDist}, are devoted to the relations between Patterson–Sullivan distributions and invariant Ruelle respectively Wigner distributions.

 \subsection*{Notation} For a set $X$, we write $\mathrm{Maps}(X, \C)$ for the vector space of maps $f: X \rightarrow \C$ with the pointwise operations. If $X$ carries a topology, then $\mathrm{Maps}_c(X, \C)$ denotes the subspace of maps $f$ with compact support, i.e., the closure of the set $\{x \in X \;|\; f(x) \neq 0 \}$ is compact.

 \subsection*{Acknowledgments}
 The authors would like to thank Joachim Hilgert and Tobias Weich for helpful discussions related to this project, and in particular, thank Joachim Hilgert for suggesting the project and providing an early-stage strategy.\\
 CA was supported by a research grant from the Aarhus University Research Foundation (grant no.\@ AUFF-E-2022-9-34).
 GP acknowledges support by the Deutsche Forschungsgemeinschaft (DFG, German Research Foundation) via the grant SFB-TRR 358/1 2023 - 491392403 (CRC ``Integral Structures in Geometry and Representation Theory'') and would like to thank the Isaac Newton Institute for Mathematical Sciences, Cambridge, for support and hospitality during the programme ``Geometric spectral theory and applications'' (EPSRC grant no.\@ EP/R014604/1) where the final revisions of this paper were undertaken.

 \section{Harmonic analysis on finite regular graphs} \label{sect:harmonic_analysis_graphs}
 Let $\mathfrak{G} \coloneqq(\mathfrak{X}, \mathfrak{E})$ denote a connected \emph{graph}, consisting of a set $\mathfrak{X}$ of vertices and a set $\mathfrak{E} \subseteq \mathfrak{X}^2$ of \emph{directed edges}, which is \emph{$(q + 1)$-regular} for $q\geq 1$, i.e., every vertex has $q + 1$ neighbors.
 We assume that $\mathfrak{G}$ is a \emph{simple graph}, meaning that it has \emph{no loops} (i.e., $\mathfrak{E} \cap \; \left\{(x,x) \; \middle\vert \;  x \in \mathfrak{X} \right\} = \emptyset $), has no multiple edges, and is symmetric under the switch of vertices (i.e., if $(x,y) \in \mathfrak{E}$, then $(y,x) \in \mathfrak{E}$).
 For each directed edge $\vec{e} \coloneqq(a,b)$, with $a,b \in \mathfrak{X},$ we call $\iota(\vec{e}) \coloneqq a$ the \emph{initial} and $\tau(\vec{e}) \coloneqq b$ the \emph{terminal} vertex of $\vec{e}$. Moreover, let $\eop{} \coloneqq (b,a)$ denote the \emph{opposite} edge of $\vec{e}$.

 A graph $\mathfrak{G}$ is called a \emph{tree} if there are no circuits in the graph when we identify each edge with its opposite, so that there is a unique path connecting any two vertices. Equivalently, it is a simple, convex graph without cycles. Note that $(q+1)$-regular graphs can be seen as quotients of $(q+1)$-regular trees for $q\geq 2$, also known as \emph{homogeneous trees} $\mathfrak{X}$.

 \subsection{Boundary at infinity of a graph and its topology}
 A sequence $\vec{e}_1 , \vec{e}_2, \ldots$ of directed edges is called \emph{concatenated}, if $\tau(\vec{e}_i) = \iota(\vec{e}_{i+1})$ for all $i$, and \emph{non-backtracking}, if $\tau(\vec{e}_{i + 1}) \neq \iota(\vec{e}_i)$ for all $i$.
 A concatenated, non-backtracking sequence of edges is called a \emph{chain (of edges)}. We denote the set of all chains by $\mathfrak{P}^+ $ and write $\pi^{\mathfrak{E}}(\mathbf{p}) \coloneqq \vec{e}_1$ for the first edge of a chain $\mathbf{p} = (\vec{e}_1, \vec{e}_2, \ldots) \in \mathfrak{P}^+ $.
 Moreover, we define two infinite chains to be \emph{equivalent}, if they are eventually equal, i.e., if there exists an index from which on the edges on the chains match except for a shift.
 We call the set of the corresponding equivalence classes $[\vec{e}_1 , \vec{e}_2, \ldots]$ of chains $(\vec{e}_1, \vec{e}_2, \ldots)$ the \emph{boundary at infinity} $\Omega$ of $\mathfrak{G}$.
 It is endowed with the topology consisting of the sets
 \begin{equation*}
  \partial_+ \vec{e} \coloneqq \left\{ \omega \in \Omega \;|\; \exists\, \text{chain}\,(\vec{e}, \vec{e}_1 , \ldots) \colon [\vec{e}, \vec{e}_1 , \ldots] = \omega \right\}.
\end{equation*}
If the orientation is reversed, we denote the set $\partial_- \vec{e} \coloneqq \partial_+ \eop{}$.
Hence $\Omega= \partial_+ \vec{e} \cup \partial_- \vec{e}$ for all $\vec{e} \in \mathfrak{E}$ and $\Omega= \cup_{ \iota(\vec{e})=x} \partial_+ \vec{e}$ for every $x \in \mathfrak{X}$, \cite[Rem.~3.10]{BHW22}.

For vertices $x,y \in \mathfrak{X}$ we write $[x,y]$ for the (unique) chain connecting $x$ and $y$ and, for $\omega \in \Omega $, $[x, \omega)$ for the representative of $\omega$ starting in $x$. For $\omega_1 \neq \omega_2 \in \Omega$, we denote $]\omega_1, \omega_2[$ for the (unique up to shift) geodesic from $\omega_1$ to $\omega_2$.

\begin{rem} \label{rem:proba_measure_nu}
  One can show that $\Omega$ is a compact totally disconnected topological space with the topology consisting of the basic open sets
  $\Omega(x,y) \coloneqq \{\omega \in \Omega \;|\; y \in [x,\omega) \}$ for $x,y \in \mathfrak{X},$ see \cite[§2]{BHW22}.
  Indeed, for $x\in \mathfrak{X}$ and $n \in \N_0$, the sets $\Omega(x,y)$ with $d(x,y)=n$ form a partition of $\Omega$ into $(q+1)q^{n-1}$ disjoint open and closed sets.
  Hence, for every $x\in \mathfrak{X}$, there is an unique Borel probability measure $\nu_x$ on $\Omega$ such that
  $$\nu_x(\Omega(x,y))=(q+1)^{-1}q^{1-n}, \quad \forall n\in \N.$$
  In the literature, this measure $\nu_x$ is known as the \textit{harmonic measure} on $\Omega$ viewed from $x\in \mathfrak{X}$, see \cite[§2.1]{LeMas14}.
\end{rem}

\subsection{Horocycles and regular graphs as homogeneous spaces}

\subsubsection{The horocycle bracket}
Let $o \in \mathfrak{X}$ denote a fixed base point and $(\omega_- , \omega_+) \subseteq \mathfrak{X}$ a reference geodesic containing $o$. For each pair $(x, \omega) \in \mathfrak{X} \times \Omega$ there exists a unique $y \in \mathfrak{X}$ such that $[o, \omega ) \cap [x, \omega ) = [y, \omega)$.
We define the \emph{horocycle bracket}
\begin{equation*}
  \langle \cdot , \cdot \rangle \colon \mathfrak{X} \times \Omega \rightarrow \mathbb{Z} , \quad \langle x, \omega \rangle \coloneqq d(o, y) - d(x, y),
\end{equation*}
where $d(\cdot,\cdot)$ denotes the metric which assigns the minimal length of chains $\vec{e}_1 , \ldots, \vec{e}_{\ell}$ connecting a pair of vertices, i.e., $d(x,y)=l([x,y])$ for any vertices $x$ and $y$. We then denote by $H_\omega(x)\coloneqq\{ y \in \mathfrak{X} \;|\; \langle x, \omega \rangle = \langle y, \omega \rangle \}$ the \emph{horocycle} passing through $x\in \mathfrak{X}$ and $\omega \in \Omega$. 

\subsubsection{Homogeneous spaces}\label{sect:homogeneous_space}
Let $G \coloneqq \mathrm{Aut}(\mathfrak{G})$ denote the automorphism group of the graph $\mathfrak{G}$ and $K \coloneqq \mathrm{Stab}_{G}(o)$ denote the stabilizer of $o$ in $G$.
We can then view $\mathfrak{X}$ as the \emph{$G$-homogeneous space} $G/K$.
Note that the horocycle bracket is invariant under the diagonal action of $K$. Moreover, if
\begin{equation*}
  B_{\omega _{\pm}} \coloneqq \left\{ \gamma \in \operatorname{Aut}(\mathfrak{X}) \;|\; \gamma(\omega _{\pm}) = \omega _{\pm} ,\ \exists \,x \in \mathfrak{X} \colon \gamma(x) = x \right\},
\end{equation*}
and we fix a $1$-step shift $\tau$ on the geodesic $(\omega_- , \omega_+)$, i.e., $\tau \in \operatorname{Aut}(\mathfrak{X})$ that acts on $(\omega_- , \omega_+)$ by mapping each vertex to one of its neighbors, we can decompose each $g \in G$ as $g=kn \tau^j \in KB_{\omega _{+ }}\langle \tau \rangle$ (``\textit{Iwasawa decomposition}'') for a unique $j \eqqcolon H(g) \in \mathbb{Z}$ and the horocycle bracket is given by
\begin{equation}\label{eq:horocycle_Hg}
  \langle go, k \omega_+  \rangle=- H(g^{-1} k),
\end{equation}
see \cite[Cor.\@ 2.8, Lem.\@ 2.9]{AFH23Pairing}. Here $\langle \tau \rangle =\{\tau^j \;|\; j \in \Z\}$ is the analogue of $A$ in the representation theory and $B_{\omega_+}$ corresponds to $N$. Note that $\langle \tau \rangle$ is abelian and isomorphic to the group $\Z$.
We denote the $K$-projection, which is unique up to right multiplication by $K \cap B_{\omega_+}$, by $k(g)$.
Let us collect some properties of the horocycle bracket.

\begin{lem}\label{lem:brackets_propI}
  Let $g, g' \in G,\ x \in \mathfrak{X},\ k \in K$ and $\omega \in \Omega$. Then
  \begin{enumerate}[label=(\roman*)]
  \item\label{eq:horocycle_identity} $\langle gx, g \omega \rangle = \langle x , \omega \rangle + \langle go, g \omega \rangle$ (horocycle identity),
  \item\label{eq:minus_brackets} $\langle go , g \omega \rangle =  - \langle g^{-1} o, \omega \rangle$,
  \item\label{eq:Hg_brackets} $\langle go, g \omega_+ \rangle = H(g) = - \langle g^{-1} o, \omega_+ \rangle$,
  \item $H(g g' k) = H(gk(g'k)) + H(g'k)$.
  \end{enumerate}
\end{lem}

\begin{proof}
  The first part follows from \cite[(15)]{BHW23}, the second from \cite[Lem.~2.10]{AFH23Pairing} or \cite[Lem.~5.15]{AFH23} and \ref{eq:Hg_brackets} is a direct consequence of \eqref{eq:horocycle_Hg} and \ref{eq:minus_brackets}. For the last part, let $g'k = k_0 n_0 \tau^{j_0}$ and $gk_0 = k_1 n_1 \tau^{j_1}$ for some $k_0, k_1 \in K,\ n_0, n_1 \in B_{\omega_+}$ and $j_0, j_1 \in \mathbb{Z}$.
  We claim that $\tau^{j_1} n_0 \tau^{-j_1} \in B_{\omega_+ }$. Indeed, since $n_0 \in B_{\omega_+}$, there exists a vertex $x \in \mathfrak{X}$ such that $n_0(x) = x$ and $n_0(\omega_+) = \omega_+$. Thus,
  \begin{align*}
    \tau^{j_1}n_0 \tau^{-j_1}(\omega_+) = \tau^{j_1}n_0(\omega_+) = \tau^{j_1}(\omega_+) = \omega_+ \quad \text{and}\quad  \tau^{j_1}n_0 \tau^{-j_1}(\tau^{j_1}x) =  \tau^{j_1}n_0(x) = \tau^{j_1}x.
  \end{align*}
  This implies the claim and the well-definedness of $H(g k(g'k))$ and allows us to write
  \begin{gather*}
    g g'k=g k_0 n_0 \tau^{j_0}=k_1 n_1 \tau^{j_1}n_0 \tau^{j_0}=k_1 (n_1 \tau^{j_1}n_0 \tau^{- j_1})\tau^{j_0 + j_1 } \in K B_{\omega_+} \langle \tau \rangle.\qedhere
  \end{gather*}
\end{proof}

Let $r \in K$ be an element with $r^2 = \mathrm{id}$ and $r \tau^j r^{-1} = \tau^{-j}$ for each $j \in \mathbb{Z}$. Then we have the following lemma.
\begin{lem}\label{lem:brackets_propII}
  Let $g, g' \in G$. Then
  \begin{enumerate}[label=(\roman*)]
  \item $\langle go, g \omega_- \rangle = H(g r) = - \langle g^{-1} o, \omega_- \rangle$,
  \item $H(g'g) = H(g) + \langle g'o, g'g \omega_+\rangle$,
  \item $H(g'gr) = H(gr) + \langle g'o,g'g \omega_- \rangle$.
  \end{enumerate}
\end{lem}

\begin{proof}
  For the first part note that $\langle go, g \omega_-  \rangle= \langle gro, gr \omega_+ \rangle$ and apply Lemma~\ref{lem:brackets_propI}\ref{eq:Hg_brackets}. Moreover,
  Lemma~\ref{lem:brackets_propI}\ref{eq:Hg_brackets} and \ref{eq:horocycle_identity} imply
  \begin{equation*}
    H(g'g) = \langle g'go,g'g \omega_+  \rangle = \langle go, g \omega_+ \rangle + \langle g'o,g'g \omega_+ \rangle = H(g) + \langle g'o,g'g \omega_+ \rangle.
  \end{equation*}
    The last part follows from the second one applied to $gr$ instead of $g$.
  \end{proof}

  \subsubsection{Open cells and measures}
  Recall from \cite[Prop.\@ A.5]{AFH23Pairing} the \emph{Bruhat decomposition}
  \begin{equation*}
  G = B_{\omega _{\pm}}B_{\omega_{\mp}} \langle \tau \rangle \sqcup r \mathrm{Stab}_G(\omega _{\pm}),
\end{equation*}
where $\mathrm{Stab}_G(\omega _{\pm}) \coloneqq \left\{ g \in G \mid g \omega _{\pm} = \omega _{\pm} \right\}$ denotes the stabilizer of $\omega _{\pm}$ in $G$. We call $B_{\omega _{\pm}}B_{\omega_{\mp}} \langle \tau \rangle$ the \emph{open cell}. It has full measure in $G$ (see \cite[Lem.\@ A.8]{AFH23Pairing}). Moreover, denoting
\begin{equation*} \label{eq:M}
  M \coloneqq \left\{ \gamma \in K \mid \gamma |_{(\omega_- , \omega_+)} = \mathrm{id} \right\} = \mathrm{Stab}_G(o) \cap \mathrm{Stab}_G(\omega_-) \cap \mathrm{Stab}_G(\omega_+),
\end{equation*}
we have the following proposition.

\begin{prop}\label{prop:open_cells}
  The map
  \begin{equation*}
    \phi \colon G / M \langle \tau \rangle \rightarrow(\Omega \times  \Omega) \backslash \mathrm{diag}(\Omega),\quad g M \langle \tau \rangle \mapsto (g \omega_-, g \omega_+)
  \end{equation*}
  is a $G$-equivariant bijection with open, dense image, where $G$ acts diagonally on $\Omega \times \Omega.$
\end{prop}

\begin{proof}
  The first part follows from \cite[Prop.\@ 2.11]{AFH23Pairing}. Moreover, $(\Omega \times \Omega) \backslash \mathrm{diag}(\Omega)$ is open.
  Indeed, let $\omega_1 \neq \omega_2 \in \Omega $. Then there exists a vertex $x \in \mathfrak{X}$ such that $[x, \omega_1) \neq [x, \omega_2)$. But if $y \in \mathfrak{X}$ denotes the last common vertex of $[x, \omega_1)$ and $[x, \omega_2)$ and $\vec{e}_1, \vec{e}_2$ denote the unique edges with $\iota(\vec{e}_1) = \iota(\vec{e}_2) = y$ and $\omega_1 \in \partial_+ \vec{e}_1 , \, \omega_2 \in \partial_+ \vec{e}_2 ,$ then $\partial_+ \vec{e}_1 \times \partial_+ \vec{e}_2$ is an open neighborhood of $]\omega_1, \omega_2[$ in $(\Omega \times \Omega) \backslash \mathrm{diag}(\Omega)$.
\end{proof}

Let
\begin{equation*}
  \mathfrak{P} \coloneqq \left\{(\mathbf{p}_1 , \mathbf{p}_2) \in ({\mathfrak{P}^{+}})^2 \mid \pi^{\mathfrak{E}}(\mathbf{p}_1) \neq \pi^{\mathfrak{E}}(\mathbf{p}_2),\, \iota(\pi^{\mathfrak{E}}(\mathbf{p}_1)) = \iota(\pi^{\mathfrak{E}}(\mathbf{p}_2)) \right\}
\end{equation*}
denote the space of bi-infinite chains with a distinguished vertex (given by $\iota(\pi^{\mathfrak{E}}(\mathbf{p}_1))$). Note that $\mathfrak{P}$ embeds canonically into $(\Omega \times \Omega) \backslash \mathrm{diag}(\Omega) \times \mathfrak{X}$ via $(\mathbf{p}_1,\mathbf{p}_2) \mapsto ([\mathbf{p}_1], [\mathbf{p}_2], \iota(\pi^{\mathfrak{E}}(\mathbf{p}_1)))$. We endow $\mathfrak{P}$ with the corresponding subspace topology. Then we have

\begin{prop} \label{prop:psi_hom}
  The map
  \begin{equation*}
    \psi \colon G / M \rightarrow \mathfrak{P},\quad g M \mapsto \left( g \omega_- , g \omega_+ , go \right)
  \end{equation*}
  is a homeomorphism.
\end{prop}

\begin{proof}
  By \cite[§4.1]{AFH23Pairing}, $\psi$ is bijective. In order to show that it is a homeomorphism, we first note that a basis of open neighborhoods of $g_0$ in $G$ is given by the sets
  \begin{equation*}
    U_F(g_0) \coloneqq \left\{ h \in G \mid \forall x \in F \colon g_0 x = h x \right\},
  \end{equation*}
where $F = \left\{ y_1 , \ldots , y_n \right\}$ denotes a finite subset of $\mathfrak{X}$, see e.g.\@ \cite[I.4]{FTN91}. 
Since $U_{F \cup F_o }(g_0) \subseteq U_{F}(g_0)$ for every finite subset $F_o \subseteq \mathfrak{X}$, we may, without loss of generality, assume that $[o, y_i ] \cap [o, y_j ] = \left\{ o \right\}$ if $i \neq j$ and that $y_1 = o$. 
But then the open set 
$$\partial_+(g_0 z_1 , g_0 y_1) \times \partial_+(g_0 z_2 , g_0 y_2) \times \left\{ go \right\},$$ 
where $(z_j , y_j) \in \mathfrak{E}$ denotes the last directed edge on the path $[o,y_j]$, is contained in $\psi(U_F(g_0))$.

On the other hand, if $(\mathbf{p}_1, \mathbf{p}_2) \in \mathfrak{P}$, choose $g \in G$ such that $g$ maps $o$ onto $\iota(\pi^{\mathfrak{E}}(\mathbf{p}_1))$ and two distinct neighbors $y_1 , y_2$ of $o$ onto $\tau(\pi^{\mathfrak{E}}(\mathbf{p}_1)),\tau(\pi^{\mathfrak{E}}(\mathbf{p}_2))$, respectively. 
Then 
$$\psi^{-1}(\partial_+  \pi^{\mathfrak{E}}(\mathbf{p}_1) \times \partial_+  \pi^{\mathfrak{E}}(\mathbf{p}_2) \times \left\{ \iota(\pi^{\mathfrak{E}}(\mathbf{p}_1)) \right\})$$ 
contains the open set $U_{\left\{ o, y_1 , y_2 \right\}}(g_0)$. \end{proof}

\begin{rem}
    Note that Proposition~\ref{prop:psi_hom} is the coordinate-free version of \cite[Prop.~2.10]{HHS12} in the Archimedean case, since the choice of $o$ is not natural in the tree-context.
\end{rem}

Using the Iwasawa decomposition, we can analyse the $G$-action on the boundary
$$ G/ \mathrm{Stab}_G(\omega_+) \cong K/K \cap \mathrm{Stab}_G(\omega_+) = K/K \cap B_{\omega_+} \cong \Omega$$
of the tree. By Proposition~\ref{prop:psi_hom} we may identify $\mathfrak{X} \times \Omega$ with the space $G/M$.

\subsubsection{Normalization of measures}
In this section, we discuss how we normalize our Haar measures. First, on the compact groups $K$ and $M$ we normalize such that the total mass is one. Moreover, if $H$ is a locally compact group and $C \subseteq H$ is a closed subgroup, we normalize the left-$H$-invariant measure on $H / C$, if it exists, such that
\begin{equation*}
  \int_{
    H
  } f(h)\, \mathrm{d} h = \int_{
    H / C
  } \int_{
    C
  }f(hc)\, \mathrm{d} c \, \mathrm{d}( h C).
\end{equation*}
Moreover, $B_{\omega _{\pm}}$ is unimodular (see e.g.~\cite[Lem.~3.3]{Ve02}) and we normalize the corresponding Haar measure $\mathrm{d} n$ by $\mathrm{d} n(B_{\omega _{\pm}} \cap \mathrm{Stab}_G(x)) = q^{\langle x, \omega _{\pm} \rangle}$ for each $x \in \mathfrak{X}$. This normalization in particular implies that the map $n \mapsto r n r^{-1}$ from $B_{\omega_+}$ to $B_{\omega_-}$ is measure-preserving, since $r(B_{\omega_+} \cap \mathrm{Stab}_{G}(x) )r^{-1} = B_{\omega_-} \cap \mathrm{Stab}_{G}(rx)$ for each $x\in \mathfrak{X}$. Finally, the Haar measure $\mathrm{d} g$ on $G$ can be normalized such that
\begin{equation}\label{eq:Haarmeasure_int}
  \int_{
    G
  } f(g) \mathrm{d} g = \int_{
    B_{\omega _{+ }}
  }\sum_{
    j \in \mathbb{Z}
  }\int_{
    K
  } f(n \tau^j k) q^{- j} \mathrm{d} k \mathrm{d}n,
\end{equation}
see \cite[Thm.~3.5]{Ve02}.
In addition, by \cite[Lem.~3.8]{Ve02}, the subgroup $\langle \tau \rangle$ acts on $B_{\omega_+}$ by conjugation and for compactly supported functions $f$ on $B_{\omega_+}$, we have
\begin{equation}\label{eq:B_omega+}
  \int_{B_{\omega_+}} f(\tau^j n \tau^{-j}) \; \mathrm{d}n = q^{-j} \int_{B_{\omega_+}} f(n) \; \mathrm{d}n.
\end{equation}

\subsection{Boundary values and Poisson transforms}
\label{sect:boundary_values_Poisson_transform}
In this section, for a spectral parameter $s\in \C$, we introduce $\chi(s)$\emph{-boundary value measures} on the boundary of a tree associated with Laplace eigenfunctions (of moderate growth) and study their regularity properties. 

\subsubsection{Poisson transforms for trees} \label{sect:Poisson_transform}
Consider the algebraic dual space
$
\mathcal{D}'(\Omega) \coloneqq \left(C^{\mathrm{lc}}_c(\Omega)\right)'
$
of the space of compactly supported, locally constant complex-valued functions on $\Omega$:
$$
C^{\mathrm{lc}}_c(\Omega) \coloneqq \left\{ f : \Omega \to \C \;\middle|\; f \text{ is locally constant with compact support} \right\}.
$$
By \cite[Prop.\@ 3.9]{BHW22}, we known that  $\mathcal{D}'(\Omega)$ is canonically isomorphic to the space of complex-valued \emph{finitely additive measures} on the boundary $\Omega$ of a tree:
$$
\mathcal{M}_{\mathrm{fa}}(\Omega) = \left\{ \mu : \Sigma \to \C \;\middle|\; \mu(\emptyset) = 0,\; \mu(U \cup U') = \mu(U) + \mu(U') \text{ for all disjoint } U, U' \in \Sigma \right\},
$$
where $\Sigma$ denotes the set of clopen (i.e., both open and closed) subsets of $\Omega$.

This identification enables us to view elements of $\mathcal{D}'(\Omega)$ as finitely additive measures on $\Omega$, making it possible to define integral transforms such as the Poisson transform in the context of trees.

\begin{definition} \label{def:Poisson_transform}
  For $s \in \C$, the (scalar) \emph{Poisson transform} $\mathcal{P}_s$ is defined as
  \begin{equation*}
    \mathcal{P}_s \colon \mathcal{D} '(\Omega) \rightarrow \mathrm{Maps}(\mathfrak{X}, \mathbb{C}), \quad \mu \mapsto \int_{
      \substack{
        \Omega
      }
    }p_s(\cdot, \omega) \mathrm{d}\mu(\omega),
  \end{equation*}
  where $p_s$ denotes the \emph{Poisson kernel}
  \begin{equation*}
    p_s \colon \mathfrak{X} \times \Omega \rightarrow \mathbb{C} , \quad(x, \omega) \mapsto q^{(\frac{1}{2}+is)\langle x, \omega \rangle}.
  \end{equation*}
\end{definition}

\begin{rem}[Convention]
    Note that the $s$ from \cite[Rem.~4.17]{AFH23Pairing} has, compared to ours, a factor of $i$, i.e., our $\frac{1}{2} + i s$ corresponds there to $\frac{1}{2} + i \cdot i s = \frac{1}{2} - s.$ This has the advantage that this parametrization fits the spectrum better.
\end{rem}

Consider now the (vertex) \textit{Laplace operator} $\Laplace$ on $\mathfrak{G}$ given by
$$(\Laplace f)(x) \coloneqq \frac{1}{q + 1}\sum_{
    d(x,y)=1} f(y), \quad f \in \mathrm{Maps}(\mathfrak{X}, \C),$$
which operates on $\mathrm{Maps}(\mathfrak{X}, \C)$ by averaging over neighbors, \cite[§3]{BHW23}. In the literature, it is also referred to as the \emph{stochastic operator} or the \emph{average of the values of $f$} on the nearest neighbors of $x$ and is related to the discrete Laplacian \cite{ALe15} (see also \cite[Rem.~1.1]{AFH23} for a variety of different Laplacians on graphs).\\
For $s \in \mathbb{C}$,
let
$$\chi(s) \coloneqq  \frac{\sqrt{q}}{q + 1}(q^{is} + q^{- is})= \frac{2\sqrt{q}}{q+1} \cos (s \ln(q))$$
be the potential and consider its associated eigenspace
\begin{equation*}
  \mathcal{E}_{\chi(s)}(\Laplace ; \mathrm{Maps}(\mathfrak{X}, \mathbb{C})) \coloneqq \left\{ f \colon \mathfrak{X} \rightarrow \mathbb{C} \mid \forall \, x \in \mathfrak{X}\colon \Laplace f(x) = \chi(s) f(x) \right\}.
\end{equation*}
Note that, as $s$ varies over $\mathbb{R}$, the eigenvalue
$\chi(s) \in \Big[-\frac{2\sqrt{q}}{q+1}, \frac{2\sqrt{q}}{q+1}\Big]$, which corresponds to the \emph{tempered spectrum}. Following \cite{ALe15}, one usually restricts $s \in [0, \frac{\pi}{\ln(q)}]$. On the other hand, if $s\in i\big(-\frac{1}{2},\frac{1}{2}\big)+\mathbb{Z} \frac{\pi}{\ln(q)}$,
then $\chi(s) \in [-1,1] \backslash \big(-\frac{2\sqrt{q}}{q+1}, \frac{2\sqrt{q}}{q+1}\big)$, which corresponds to the \emph{untempered spectrum}. In this paper, we will focus exclusively on the tempered part of the spectrum.\\
The following result shows that the Poisson transform provides a linear isomorphism between generalized functions on the boundary $\Omega$ and eigenfunctions of the vertex Laplacian.

\begin{thm}[{\cite[Thm.\@ 4.7]{BHW22}}]\label{thm:Poissontransform_delta}
  For $\chi(s) \notin \left\{ \pm 1 \right\}$, the Poisson transform maps $\mathcal{D}'(\Omega)$ onto $\mathcal{E}_{\chi(s)}(\Laplace ; \mathrm{Maps}(\mathfrak{X},\mathbb{C}))$ and the resulting map
  \begin{equation*}
    \mathcal{P}_s \colon \mathcal{D} '(\Omega) \rightarrow \mathcal{E}_{\chi(s)}(\Laplace ; \mathrm{Maps}(\mathfrak{X}, \mathbb{C}))
  \end{equation*}
  is a linear isomorphism. The inverse
  \begin{equation*}
    \mathcal{P}_s^{-1} \colon \mathcal{E}_{\chi(s)}(\Laplace ; \mathrm{Maps}(\mathfrak{X}, \mathbb{C})) \rightarrow \mathcal{D}'(\Omega ), \quad \phi \mapsto \mathcal{P}_s^{-1}(\phi),
  \end{equation*}
  is uniquely determined by
  \begin{equation*}
    \mathcal{P}_s^{-1}(\phi)(\mathbbm{1}_{\partial_+ \vec{e}}) \coloneqq
    \begin{cases}
      q^{-(\frac{1}{2} + is)d(o, \iota(\vec{e}))} \; \frac{\phi(\tau(\vec{e})) - q^{- \frac{1}{2} - is}\phi(\iota(\vec{e}))}{q^{\frac{1}{2} +  is } - q^{- \frac{1}{2} - is}}      & \colon \vec{e} \text{ points away from }o \\
      \sum_{
      \vec{e}\,' \in \mathfrak{E} \colon \iota(\vec{e}\,') = o
      }\mathcal{P}_s^{-1}(\phi)(\mathbbm{1}_{\partial_+ \vec{e}\,'})- \mathcal{P}_s^{-1}(\phi)(\mathbbm{1}_{\partial_+ \eop{}})        & \colon \text{else}.
    \end{cases}
  \end{equation*}
    \end{thm}

\noindent
For geometric reasons we call the inverse $\mathcal{P}^{-1}_s$ of the Poisson transform the $\chi(s)$-\textit{boundary value map} \cite{BHW22}, which we denote by $$\mu_{s, \phi} \coloneqq \mathcal{P}_s^{-1}(\phi )\in \mathcal{D}'(\Omega) \quad \text{ for } \phi \in \mathcal{E}_{\chi(s)}(\Laplace ; \mathrm{Maps}(\mathfrak{X}, \mathbb{C})).$$
Moreover, if 
$\phi \in \mathcal{E}_{\chi(s)}(\Delta; \mathrm{Maps}(\mathfrak{X}, \mathbb{C}))$, 
then its complex conjugate 
$\overline{\phi}$ lies in the space
$\mathcal{E}_{\chi(-\overline{s})}(\Delta; \mathrm{Maps}(\mathfrak{X}, \mathbb{C}))$.  
By Theorem~\ref{thm:Poissontransform_delta}, there exist unique distributions 
$\mu_1, \mu_2 \in \mathcal{D}'(\Omega)$ such that
$$
\phi = \mathcal{P}_s(\mu_1) \quad \text{and} \quad 
\overline{\phi} = \mathcal{P}_{-\overline{s}}(\mu_2).
$$
On the other hand, note that
\begin{equation} \label{eq:complex_conj}
\overline{\phi}(x) 
= \overline{\mathcal{P}_s(\mu_1)(x)} 
= \overline{\mu_1}\Bigl(q^{\left(\frac{1}{2} - i \overline{s}\right) \langle x, \cdot \rangle}\Bigr) 
= \mathcal{P}_{-\overline{s}}(\overline{\mu_1})(x),
\end{equation}
where, for any distribution $\mu$, the complex conjugate is defined by
$
\overline{\mu}(\phi) := \overline{\mu(\overline{\phi})}.
$
Comparing with the previous representation, we conclude that 
$\mu_2 = \overline{\mu_1}$, i.e.,
$$
\mathcal{P}_{-\overline{s}}^{-1}(\overline{\phi}) = 
\overline{\mathcal{P}_s^{-1}(\phi)}.
$$

These $\chi(s)$-boundary values play the same conceptual role as \emph{Helgason boundary values} in the symmetric space setting \cite{AZ07, HHS12, DP24}.
In the latter case, the moderate growth condition of the eigenfunctions is essential to ensure both existence and uniqueness of the Helgason boundary distribution, since the Poisson transform is defined on increasingly large classes of generalized boundary functions and uncontrolled growth may prevent the boundary value from being realized as a distribution.\\
For homogeneous trees, however, the Poisson transform yields a linear isomorphism between boundary distributions, realized as finitely additive measures, and Laplace eigenfunctions, so boundary values exist and are unique without imposing any growth condition.
Growth assumptions are only needed to obtain additional regularity properties of the boundary distributions, see Section~\ref{sect:boundaryvalues_regularity}.

\begin{rem}[Intertwining properties of the Poisson transform]
  For each \(s \in \mathbb{C}\), consider the function
  $
  N_s \colon G \times \Omega \rightarrow \mathbb{C},\ (g, \omega) \mapsto N_s(g, \omega) \coloneqq q^{-(\frac{1}{2} + is) \langle go, g\omega \rangle},
  $
  and define a representation \(\pi_s\) of \(G\) on \(\mathcal{D}'(\Omega)\) by
  \begin{equation} \label{eq:rep_pis}
      \pi_s(g)\mu \coloneqq g_*\big(N_s(g, \cdot) \mu\big) = N_s(g, \cdot)\mu(g^{-1} \cdot),
    \end{equation}
    where \(N_s(g, \cdot)\mu\) denotes multiplication of the distribution \(\mu\) by the (locally constant) function \(N_s(g, \cdot)\), and \(g_*\) denotes the pushforward under the automorphism \(g\) (see \cite[§11]{BHW23} and {\cite[§2.4]{AFH23}} for more details).\\
    This representation \(\pi_s\) is the analogue of the \emph{spherical principal series representation} in the setting of symmetric spaces: just as in the Lie group context, it arises from inducing a character on a boundary (or minimal parabolic) to the group \(G\), and the Poisson transform intertwines this boundary representation with the regular representation on functions on the space $\mathfrak{X}$.
    Indeed, as shown in \cite[Obs.~11.2]{BHW23}, the Poisson transform intertwines $\pi_s$ with the left regular representation $l_g$ of $G$ on $\mathcal{E}_{\chi(s)}(\Laplace; \mathrm{Maps}(\mathfrak{X}, \mathbb{C}))$:
    \begin{eqnarray*}
        \mathcal{P}_s(\pi_s(g)\mu)(x)
        = \int_\Omega p_s(x, \omega) \, \mathrm{d}(\pi_s(g)\mu)(\omega)
        &=& \int_\Omega p_s(x, \omega) \, N_s(g, \omega) \, \mathrm{d}\mu(g^{-1}\omega) \\
        &=& \int_\Omega q^{\left(\frac{1}{2} + is\right)(\langle x, \omega \rangle + \langle g o, g\omega \rangle)} \, \mathrm{d}\mu(g^{-1}\omega) \\
        &\stackrel{\text{Lem.}~\ref{lem:brackets_propI}~\ref{eq:horocycle_identity} + \ref{eq:minus_brackets}}{=}&\int_\Omega q^{\left(\frac{1}{2} + is\right)\langle g^{-1}x, \omega \rangle} \, \mathrm{d}\mu(g^{-1}\omega) \\
        &=& \mathcal{P}_s(\mu)(g^{-1}x) = (l_g\circ \mathcal{P}_s(\mu))(x),
    \end{eqnarray*}
    for all \(g \in G\), \(\mu \in \mathcal{D}'(\Omega)\), and \(x \in \mathfrak{X}\).
\end{rem}

In addition, by means of a limiting procedure in the case of a regular tree, one can recover the measure $\mu_{s, \phi} \in \mathcal{M}_{\mathrm{fa}}(\Omega) = \mathcal{D}'(\Omega)$ from its Poisson transform $\phi = \mathcal{P}_s(\mu) \in \mathcal{E}_{\chi(s)}(\Laplace ; \mathrm{Maps}(\mathfrak{X}, \mathbb{C}) )$:

\begin{thm}[{\cite[Cor.\@~4.12]{BHW22}}]\label{thm:measureU}
  Let $\mu \in \mathcal{D}'(\Omega)$ and $s \in \mathbb{C}$ with $\mathrm{Im}(s) < 0$. Then, for any clopen set $U=\bigcup_{x \in \mathfrak{X}_n(U)} \Omega(o,x) \subseteq \Omega$ given as a finite union of sets $\Omega(o,x)$ (see Remark~\ref{rem:proba_measure_nu}), we have
  \begin{equation*}
        \mu(U) = \frac{q^{2 is} - q^{-1} }{q^{2 is} - 1} \lim_{n \to \infty}\frac{1}{q^{n(\frac{1}{2} + is)}}\sum_{
          x \in \mathfrak{X}_n(U)
        }\underbrace{\mathcal{P}_s(\mu)(x)}_{=\phi(x)},
      \end{equation*}
      where $\mathfrak{X}_n(U) \coloneqq \left\{ x \in [o,U) \mid d(o,x) = n \right\}$, with $[o,U) \coloneqq \bigcup_{\omega \in U}[o,\omega )$, denotes the set of vertices at distance $n$ from $o$ that lie on a path from $o$ to some boundary point in $U$.
\end{thm}

\begin{rem}
    Note that in \cite[Cor.\@~4.12]{BHW22}, the condition should be $q < \lvert z^2 \rvert$ as $z$ is complex. Translating this into our notation, we have $\mathrm{Im}(s)<0$ if and only if $q < \lvert z \rvert^2 $.
\end{rem}

\subsubsection{Regularity of boundary values}
\label{sect:boundaryvalues_regularity}
To study regularity properties of the boundary values $\mu_{s,\phi}$, we restrict our attention to eigenfunctions $\phi$ of moderate growth.
Such growth assumptions do not affect existence or uniqueness of boundary values, but correspond precisely to additional regularity: eigenfunctions of moderate growth give rise to boundary distributions living on the duals of certain Banach spaces of Hölder continuous functions, as shown in \cite[§5]{BHW22}.

\begin{definition}[Functions of moderate growth, {\cite[Def.~5.12]{BHW22}}]
    We say that $f\in C(\mathfrak{X})$ is of \emph{moderate growth} if there exists $C,K>0$ such that
    $$|f(x)| \leq CK^{d(o,x)} \;\;\; \forall x \in \mathfrak{X}.$$
    We denote the space of functions of moderate growth by $\mathcal{E}^*(\mathfrak{X}).$
  \end{definition}

  We set
  $$
  \mathcal{E}^*_{\chi(s)}(\mathfrak{X}) \coloneqq \mathcal{E}^*(\mathfrak{X}) \cap \mathcal{E}_{\chi(s)}(\Laplace; \mathrm{Maps}(\mathfrak{X}, \mathbb{C})).
  $$

  These spaces can be characterized via the duals of Lipschitz spaces on the boundary $\Omega$ \cite[Thm.~5.13]{BHW22}.

  \begin{thm}
    \label{thm:Lipschitz_growth_F}
    For some $0 < \vartheta < 1$, consider the topological dual space $\mathcal{F}'_{o,\vartheta}(\Omega)$ of the space of Lipschitz continuous functions on the boundary $\Omega$
    $$
    \mathcal{F}_{o,\vartheta}(\Omega) \coloneqq \left\{ f: \Omega \rightarrow \C \;\middle|\; \exists C_f > 0 \;\text{s.t.}\; \forall \omega_1, \omega_2 \in \Omega,\; |f(\omega_1) - f(\omega_2)| \leq C_f\, d_{o,\vartheta}(\omega_1, \omega_2) \right\},
    $$
    where the metric $d_{o,\vartheta}$ on the boundary is defined by
    $$
    d_{o,\vartheta}(\omega_1, \omega_2) \coloneqq \vartheta^{d_{\max}}, \quad \text{with} \quad d_{\max} \coloneqq \sup\left\{ d(o,x) \;\middle|\; x \in [o,\omega_1) \cap [o,\omega_2) \right\}.
    $$
    \noindent
    Then the boundary value $\mu_{s, \phi} = \mathcal{P}^{-1}_s(\phi) \in \mathcal{D}'(\Omega)$ is contained in $\mathcal{F}'_{o,\vartheta}(\Omega)$ for some $0<\vartheta <1$ if and only if $\phi \in \mathcal{E}^*_{\chi(s)}(\mathfrak{X})$ for $\chi(s) \in \C \setminus \{ \pm 1 \}$.
  \end{thm}

  \begin{rem}
    Note that the space of Lipschitz continuous functions $\mathcal{F}_{o,\vartheta}(\Omega)$, equipped with the norm
    $$
    \|f\|_{o,\vartheta} \coloneqq \inf\{C_f \;|\; \text{$C_f$ is a Lipschitz constant for } f\} + \|f\|_\infty, \quad \forall f \in \mathcal{F}_{o,\vartheta}(\Omega),
    $$
    is a Banach space. Moreover, the equivalence class of the norms $\|\cdot\|_{o,\vartheta}$ is independent of the choice of base point $o \in \mathfrak{X}$ \cite[Lem.~5.8]{BHW22}. Additionally, for all $0 < \vartheta < 1$ and any base point $o \in \mathfrak{X}$, we have $C^{\mathrm{lc}}(\Omega) \subseteq \mathcal{F}_{o,\vartheta}(\Omega)$ \cite[Rem.~5.9]{BHW22}.
    Consequently, $\mathcal{F}'_{o,\vartheta}(\Omega)$ is a subset of  $\mathcal{D}'(\Omega)$.
Hence, for a measure $\mu$, the condition $\mu \in \mathcal{F}'_{o,\vartheta}(\Omega)$ reflects a regularity property of $\mu$: the closer $\vartheta$ is to $1$, the better the regularity. In other words, as $\vartheta$ increases, the space $\mathcal{F}_{o,\vartheta}(\Omega)$ becomes larger, while its dual $\mathcal{F}'_{o,\vartheta}(\Omega)$ becomes smaller \cite[p.~29]{BHW23}.\\
Note also that these spaces $\mathcal{F}'_{o,\vartheta}(\Omega)$ are precisely the functional spaces appearing in the spectral theory of transfer operators for subshifts of finite type \cite{BHW22}.
  \end{rem}

  From Theorem~\ref{thm:Poissontransform_delta} together with the previous result, we get:

  \begin{prop} \label{prop:Poissontransform_modgrowth}
    For $\chi(s) \in \mathbb{C} \setminus \{\pm 1\}$, the Poisson transform restricts to a linear isomorphism \[
      \mathcal{P}_s : \bigcup_{0 < \vartheta < 1}\mathcal{F}'_{o,\vartheta}(\Omega) \to \mathcal{E}^*_{\chi(s)}(\mathfrak{X}).
    \]
\end{prop}

Consequently, for any $\chi(s) \in \mathbb{C} \setminus \{\pm 1\}$, every eigenfunction $\phi \in \mathcal{E}^*_{\chi(s)}(\mathfrak{X})$ admits a unique representation as the Poisson transform of a Hölder continuous function $\mu_{s, \phi} \in \mathcal{F}'_{o,\vartheta}(\Omega)$ for some $0<\vartheta<1$,
\begin{equation} \label{eq:boundary_values}
  \phi = \mathcal{P}_s(\mu_{s,\phi}){=  \int_\Omega q^{\left(\frac{1}{2} + i s\right) \langle \bullet , \omega \rangle} \, \mathrm{d}\mu_{s,\phi}(\omega)}.
\end{equation}

Moreover, using \cite[Lem.~5.11]{BHW22}, we derive a characterization for the
regularity of $\mu_{s,\phi}$.

\begin{prop}[Regularity of $\chi(s)$-boundary values] \label{prop:boundary_dist_regularity}
    For each vertex $x\in \mathfrak{X}$, consider the set $\Omega_o(x)=\{\omega \in \Omega \;|\; x \in [o,\omega [\} \subseteq \Omega$ consisting of all boundary points whose geodesic ray starts with the unique geodesic from $o$ to $x$.
    Let $\mu_{s,\phi} \in
\mathcal{F}'_{o,\vartheta}(\Omega)$ for some $\vartheta >0$ such that it satisfies
    \begin{equation} \label{eq:mu_Omega_bound}
        |\mu_{s,\phi}(\Omega_o(x))| \leq C'(1+\vartheta^{1-d(0,x)}) \leq CK^{d(o,x)}
    \end{equation}
    for $K>\vartheta^{-1}$ and $C,C' >0$.
    Then for $f \in \mathcal{F}_{o,\vartheta}(\Omega)$ and $K\vartheta q < 1$ we have
    $$|\mu_{s,\phi}(f)| \leq C_{K,\vartheta,q} \; \|f\|_{0, \vartheta}$$
    where $C_{K,\vartheta,q} \coloneqq C \Big(\frac{1}{1-K q} + \frac{q+1}{1-K \vartheta q}\Big).$
\end{prop}

\begin{proof}
    We base our construction on the proof of \cite[Lem.~5.11]{BHW22}. For each vertex $x \in \mathfrak{X}$, fix a representative boundary point $\omega_x \in \Omega_o(x)$. For $n \in \N_0$ and $f \in \mathcal{F}_{o,\vartheta}(\Omega)$, define the locally constant approximation
    $$
    \mathcal{W}_{n}(f) \coloneqq \sum_{x \in \mathfrak{X} \colon d(o,x)=n} f(\omega_x)\,\mathbbm{1}_{\Omega_o(x)} \in \mathcal{F}^{\mathrm{loc}}_{o,\vartheta}(\Omega).
    $$
Moreover, for $\omega \in \Omega_o(x)$ with $d(o,x)=n$, we have $\mathcal{W}_{n}(f)(\omega)=f(\omega_x).$
    As in \cite[Lem.~5.11]{BHW22}, the sequence $\mathcal{W}_{n}(f)$ converges uniformly to $f$ as $n \rightarrow \infty$.
    By continuity of $\mu_{s,\phi}$ on $\mathcal{F}_{o,\vartheta}(\Omega)$, we have
    $$
    \mu_{s,\phi}(f) = \lim_{n \rightarrow \infty} \mu_{s,\phi}(\mathcal{W}_{n}(f)) = \lim_{n \rightarrow \infty} \sum_{x \in \mathfrak{X} \colon d(o,x)=n} f(\omega_x)\,\mu_{s,\phi}(\Omega_o(x)).
    $$
    Fix a reference point $\omega_0 \in \Omega$. Then
    $ f(\omega_x) = f(\omega_0) + (f(\omega_x)-f(\omega_0)),$
    and thus
    \begin{eqnarray*}
      |\mu_{s,\phi}(\mathcal{W}_n(f))|
      &\le& |f(\omega_0)| \sum_{x \colon d(o,x)=n} |\mu_{s,\phi}(\Omega_o(x))| + \sum_{x \colon d(o,x)=n} |f(\omega_x)-f(\omega_0)|\,|\mu_{s,\phi}(\Omega_o(x))|\\
      &\eqqcolon& S_1 + S_2.
    \end{eqnarray*}
    By the growth bound \eqref{eq:mu_Omega_bound} and grouping vertices by level \(n = d(o,x)\), we obtain for $Kq <1$
    $$
    S_1 \le \|f\|_\infty \sum_{n=0}^\infty \sum_{x: d(o,x)=n} C K^n \le C \|f\|_\infty \sum_{n=0}^\infty (K q)^n = \frac{C}{1-Kq} \|f\|_\infty.
    $$
    For the second sum, using the $\vartheta$-Hölder property of $f$
    $$
    |f(\omega_x)-f(\omega_0)| \le |f|_{o,\vartheta} \vartheta^{d(o,x)}
    $$
    and the combinatorial estimate from \cite[Lem.~5.11]{BHW22} gives
    $$
    S_2 \le (q+1) |f|_{o,\vartheta} \sum_{n=0}^\infty \sum_{x: d(o,x)=n} C (K \vartheta)^n \le (q+1) C |f|_{o,\vartheta} \sum_{n=0}^\infty (K \vartheta q)^n = C\frac{q+1}{1-K \vartheta q} |f|_{o,\vartheta}
    $$
    for $K\vartheta q <1.$
    Combining the two estimates, we obtain the final bound
    $$
    |\mu_{s,\phi}(f)| \le \frac{C}{1-K q} \|f\|_\infty + C\frac{q+1}{1-K \vartheta q} |f|_{o,\vartheta} \le C \Big(\frac{1}{1-K q} + \frac{q+1}{1-K \vartheta q}\Big) \|f\|_{o,\vartheta}.
    $$
  \end{proof}

  \subsection{From trees to finite regular graphs}
  \label{sect:finite_graphs}
As in the Archimedean setting, we aim to work with finite $(q+1)$-regular graphs, which serve as the analogue of compact locally symmetric spaces of non-positive curvature.

For this we first relate a connected $(q+1)$-regular graph  $\mathfrak{G} = (\mathfrak{X}, \mathfrak{E})$
  to a tree.
Its universal cover,
denoted $\widetilde{\mathfrak{G}} = (\widetilde{\mathfrak{X}}, \widetilde{\mathfrak{E}})$, is constructed by first taking the simply connected cover of the undirected version of $\mathfrak{G}$ (identifying each edge with its opposite) and then replacing each undirected edge by a pair of oriented edges in opposite directions, thereby endowing the cover with an oriented edge structure.
  In this way, $\widetilde{\mathfrak{G}}$ is a $(q+1)$-regular tree: it is infinite, simply connected (contains no cycles), and every vertex has degree $q+1$.
  Moreover, $\widetilde{\mathfrak{G}}$ is $0$-hyperbolic in the sense of Gromov, meaning that all triangles are ``thin''.

Let $\Gamma \leq G = \mathrm{Aut}(\widetilde{\mathfrak{X}})$ be the group of \emph{deck transformations} associated with the covering map $\pi: \widetilde{\mathfrak{X}} \rightarrow \mathfrak{X}$ so that $\pi \circ \gamma = \pi$ for all $\gamma \in \Gamma$.
  Note that the action of $G$ on $\widetilde{\mathfrak{X}}$ extends to actions on $\widetilde{\mathfrak{E}}$ and $\widetilde{\mathfrak{P}}$ by acting on each vertex contained in the edge, respectively chain.
  Then $\Gamma$ is discrete and acts freely (no vertex is fixed by a nontrivial element) on the tree $\widetilde{\mathfrak{G}}$.
  The quotient by this action yields a connected, \emph{finite $(q+1)$-regular graph} without dead ends
  that can be identified with the graph $\mathfrak{G}$ we started out with. We denote this quotient by
\[
    \widetilde{\mathfrak{G}}_\Gamma = (\widetilde{\mathfrak{X}}_\Gamma, \widetilde{\mathfrak{E}}_\Gamma).
  \]
  Since the quotient has finitely many vertices, the action of $\Gamma$ on $\widetilde{\mathfrak{G}}$ is cocompact.
  From now on, we will not distinguish the universal cover by a tilde.

  \subsubsection{Phase spaces}
  \label{sect:phase_spaces}
  On a tree, the \emph{phase space} is defined as $\mathfrak{X} \times \Omega \times  \C$. Its dynamic is given by the shift operator
  \begin{eqnarray*}
    \sigma: \mathfrak{X} \times \Omega \times \C &\rightarrow& \mathfrak{X} \times \Omega \times \C \\
    (x,\omega,s) &\mapsto& \sigma(x,\omega,s) = \sigma(x=x_0,x_1,x_2, \dots, s)= (x_1,x_2, \dots, s)=(x_1, \omega, s).
  \end{eqnarray*}

  For a finite $(q+1)$-regular graph $\mathfrak{G}_\Gamma$, the automorphism subgroup $\Gamma$ acts naturally on the space $\mathfrak{P}^+$ of semi-geodesics on its universal cover $\mathfrak{X}$, and hence on $\mathfrak{X} \times \Omega$.
  A point $[x, \omega )$
in the space $\mathfrak{P}^+$ corresponds to a semi-geodesic trajectory starting at the vertex $x$ and pointing towards the boundary point $\omega$. The quotient under the $\Gamma$-action then defines the \emph{phase space} of the graph $\mathfrak{G}_\Gamma$, and we have the natural identification
  $
  \Gamma \backslash (\mathfrak{X} \times \Omega) \cong  \mathfrak{P}^{+}_\Gamma,
  $
where $\mathfrak{P}^{+}_\Gamma \coloneqq \Gamma \backslash \mathfrak{P}^{+}$ denotes the space of one-sided chains on the graph.
  The phase space can be viewed as the discrete analogue of the \emph{unit sphere bundle} in the Archimedean setting, which motivates the notation $S\mathfrak{X}_\Gamma$:
  \begin{equation*} 
    S\mathfrak{X}_\Gamma \coloneqq \Gamma \backslash (\mathfrak{X} \times \Omega).
  \end{equation*}
  {The dynamics on $\mathfrak{G}_\Gamma $, given by the \emph{non-backtracking shift} $\sigma$ on infinite paths, naturally descends to a map on the phase space:}
  $$
  \sigma([x,\omega]) = \sigma([x_0,x_1,x_2,\dots,\omega]) = [x_1, x_2, \dots, \omega] = [x',\omega],
  $$
  where $x'$ denotes the next vertex along the geodesic from $x$ to $\omega$.

  Functions and distributions on $\mathfrak{X}_\Gamma$ and $S\mathfrak{X}_\Gamma$ can be naturally identified with $\Gamma$-invariant functions and distributions on $\mathfrak{X}$ and $S\mathfrak{X}$, respectively, via the canonical projection
  {$\pi$}.
  Here, we set
  \begin{equation*}
  S\mathfrak{X} \coloneqq \mathfrak{X} \times \Omega,
\end{equation*}
allowing us to treat these spaces interchangeably without distinguishing between them.

\subsubsection{Fundamental domains and cutoff functions}
\label{sect:fundamental_domain}
Let $\mathfrak{D}$ be a subgraph of ${\mathfrak{G}}$ which is a \textit{fundamental domain} for the action of $\Gamma$ on vertices ${\mathfrak{X}}$. For simplicity, we denote by the same letter $\mathfrak{D}$ a lift of that fundamental domain to $G$.

Note that we can choose a $G$-invariant measure on a quotient space $\Gamma\backslash Y$ such that
\begin{equation} \label{eq:Ginvariantmeasure}
    \int_{\Gamma\backslash Y} \sum_{\gamma \in \Gamma} f(\gamma y) \; \mathrm{d}\Gamma y= \int_Y f(y) \mathrm{d}y
\end{equation} holds.
Indeed, by \cite[Lem.~3.5.]{AFH23Pairing}, we know that for a locally compact Hausdorff space $Y$ with a continuous $G$-action and $p:Y \rightarrow \Gamma\backslash Y$ a continuous $G$-equivariant projection such that $p^{-1}(\mathfrak{D})$ is a fundamental domain for the $\Gamma$-action on $Y$, there exists a well-defined (Radon) measure $\mathrm{d}\Gamma y$ on $\Gamma\backslash Y$ which is characterised by
\eqref{eq:Ginvariantmeasure}
for each $f \in C_c(Y)$, where $\mathrm{d}Y$ is a $G$-invariant Radon measure on $Y$.

In the setting of Riemannian manifolds, Anantharaman and Zelditch \cite{AZ07, AZ12}
replaced the characteristic function of a fundamental domain by a special smooth cutoff
function on $G$ when integrating against irregular distributions.
In the framework of graphs, where integrals are replaced by sums, one can proceed
analogously by using a finitely supported $\Gamma$-adapted
locally {constant} cutoff function on $G$.

\begin{definition}[Fundamental domain cutoff function]
  \label{def:smoth_fundamental_domain_cutoff_fct}
  A \emph{locally constant fundamental domain cutoff function} (lcfd-function) $\Xi \in C_c^{\mathrm{lc}}(S \mathfrak{X})$ is a function such that
  \begin{equation*}
    \sum_{\gamma \in \Gamma} \Xi (\gamma . (x,\omega))=1 \qquad  \forall (x, \omega)\in \mathfrak{X} \times \Omega.
  \end{equation*}
  We also consider lcfd-functions on  $G/M \cong \mathfrak{P}$ using an analogous notion.
\end{definition}

\begin{rem}
  {Note that such functions can be constructed using a fundamental domain $\mathfrak{D}$ for the $G$-action on the vertex set $\mathfrak{X}$ of the tree.} For instance, consider a ${2}$-regular finite graph $\mathfrak{G}_{{2}}$. {In this case, the universal cover is a $2$-regular tree and we may choose $\mathfrak{D}$ as a set of three consecutive vertices, see Figure~\ref{fig:fundamental_domain}}. Indeed, {for each vertex $x$,} there
  exists a unique $\gamma_x \in \Gamma$ such that the vertex $\gamma_x x$ is in the fundamental domain $\mathfrak{D}$. In this situation, $\Xi= \mathds{1}_{\mathfrak{D} {\times \Omega}}.$
\end{rem}

\begin{figure}[h!]
  \centering
\begin{tikzpicture}[
    scale=1.6,
    every node/.style={circle, draw=black, line width=0.6pt, inner sep=1.6pt},
    edge/.style={line width=0.8pt, black},
    ]
    \definecolor{OIblue}{HTML}{0072B2}
    \definecolor{OIverm}{HTML}{D55E00}
    \definecolor{OIgreen}{HTML}{009E73}

\begin{scope}[xshift=0cm]
      \coordinate (A) at (0,0.9);
      \coordinate (B) at (-0.78,-0.45);
      \coordinate (C) at (0.78,-0.45);

      \draw[edge] (A) -- (B) -- (C) -- (A);

      \node[fill=OIblue] (v1) at (A) {};
      \node[fill=OIverm] (v2) at (B) {};
      \node[fill=OIgreen] (v3) at (C) {};

      \node[draw=none, circle=none, inner sep=0pt] at (0,-1.25) {\small Finite graph $\mathfrak{G}_2$};
    \end{scope}

\begin{scope}[xshift=3.2cm]
\foreach \i/\col in {0/OIblue,1/OIverm,2/OIgreen,3/OIblue,4/OIverm,5/OIgreen,6/OIblue,7/OIverm,8/OIgreen}{
        \node[fill=\col] (w\i) at (0.6*\i,0) {};
      }

\foreach \i in {0,...,7}{
        \draw[edge] (w\i) -- (w\the\numexpr\i+1\relax);
      }

\foreach \dx in {-0.55,-0.40,-0.25} { \fill (-0.10+\dx,0) circle (0.6pt); }
      \foreach \dx in {0.25,0.40,0.55}   { \fill (0.6*8+0.10+\dx,0) circle (0.6pt); }

      \draw[edge] (-0.10,0) -- (w0);
      \draw[edge] (w8) -- (0.6*8+0.10,0);
      \begin{pgfonlayer}{background}
        \node[
        draw,
        ellipse,
        line width=0.8pt,
        inner xsep=0pt,
        inner ysep=5pt,
        fill=gray!15,
        fit=(w3)(w4)(w5),
        label={[black]above:$\mathfrak{D}$}
        ] {};
      \end{pgfonlayer}

      \node[draw=none, circle=none, inner sep=0pt] at (0.6*4, -0.65) {\small Universal cover (tree)};
      \draw[->, line width=0.9pt, dashed]
      (w0) to[bend left=45]
      node[midway, above, shape=rectangle, draw=none] {$\exists! \gamma_x \in \Gamma $} (w3);
    \end{scope}
  \end{tikzpicture}
  \caption{Existence of a lcfd-function $\Xi$ by means of an example.}
  \label{fig:fundamental_domain}
\end{figure}

\begin{prop} \label{prop:cutoff_indep}
  Let $\mu \in \mathcal{D}'(S \mathfrak{X})$ be a $\Gamma$-invariant distribution and $f\in C^{\mathrm{lc}}(S \mathfrak{X})^\Gamma$ be a $\Gamma$-invariant function. Then for any $f_1, f_2 \in C_c^{\mathrm{lc}}(S \mathfrak{X})$ such that $\sum_{\gamma \in \Gamma} f_i(\gamma . (x,\omega))=f(x,\omega)$, for $i=1,2$, we have
  $$\langle f_1, \mu \rangle_{S \mathfrak{X}} =\langle f_2, \mu \rangle_{S \mathfrak{X}}.$$
\end{prop}

\begin{proof}
  Let $\Xi \in C_c^{\mathrm{lc}}(S \mathfrak{X})$ be a lcfd-function and assume that there is $f_i \in C_c^{\mathrm{lc}}(S \mathfrak{X}), i\in \{1,2\}$,  such that $\sum_{\gamma \in \Gamma} f_i(\gamma . (x,\omega))=f(x,\omega)$. Then, using $S \mathfrak{X} = \mathfrak{X} \times \Omega$,
  \begin{eqnarray*}
    \langle f_i, \mu \rangle_{S \mathfrak{X}}
    &=& \int_{\mathfrak{X} \times \Omega} \Big\{ \sum_{\gamma \in \Gamma} \Xi (\gamma . (x,\omega)) \Big\} f_i(x,{\omega}) \mu(\mathrm{d}x, \mathrm{d}\omega)  \\
    &=& \int_{\mathfrak{X} \times \Omega} \sum_{\gamma \in \Gamma} \Xi (x,\omega) f_i({\gamma^{-1}} . (x,{\omega})) \mu(\mathrm{d}x, \mathrm{d}\omega)  \\
    &=& \int_{\mathfrak{X} \times \Omega} \Xi(x) f(x,{\omega}) \mu(\mathrm{d}x, \mathrm{d}\omega) ,\\
  \end{eqnarray*}
  where we used the $\Gamma$-invariance of $\mu \in \mathcal{D}'(S \mathfrak{X})$ in the second line.
\end{proof}

\subsubsection{$\Gamma$-invariant boundary values}
\label{sect:gamma_inv_boundaryvalues}
Theorem~\ref{thm:Poissontransform_delta} also remains valid {after taking $\Gamma$-invariants}: for $s \in \C$,  the Poisson transform $\mathcal{P}_s$ induces an isomorphism
$$\mathcal{D}'(\Omega)^{\Gamma, s} \cong \mathcal{E}_{\chi(s)}(\Laplace_\Gamma  ; \mathrm{Maps}(\mathfrak{X}_\Gamma , \mathbb{C})),$$ where
$
  \mathcal{D}'(\Omega)^{\Gamma, s} \coloneqq \left\{ \mu \in \mathcal{D}'(\Omega) \mid \forall \, \gamma \in \Gamma \colon \pi_s(\gamma) \mu = \mu \right\},
$
and $\pi_s$ is the representation of $\mathrm{Aut}(\mathfrak{G})$ defined in \eqref{eq:rep_pis},
see \cite[Rem.\@ 3.4]{AFH23} for more details. Hence, the $\chi(s)$-boundary values of
$\Gamma$-invariant eigenfunctions are themselves $\Gamma$-invariant.

\begin{lem} \label{lem:boundaryvalues_inv}
  For $s\in \C$ and $\phi \in \mathcal{E}_{\chi(s)}(\Laplace; \mathrm{Maps}(\mathfrak{X}, \mathbb{C}))^\Gamma$, the associated boundary measure {is} $\Gamma$-invariant:
  $${\forall \gamma \in \Gamma \colon \qquad \pi_s(\gamma) \mu_{s, \phi} = \mu_{s, \phi}}.$$
\end{lem}

\begin{proof}
  Since $\phi \in {\mathcal{E}_{\chi(s)}(\mathfrak{X})}$ is $\Gamma$-invariant, i.e., \@ $\phi(\gamma x) = \phi(x)$ for $\gamma \in \Gamma$, we have by the uniqueness of the boundary value
  \begin{eqnarray*}
        \phi(x)
        =\int_\Omega q^{(\frac{1}{2}+is)\langle \gamma  x, \omega \rangle} \mu_{s, \phi}(\mathrm{d}\omega)
    &=&\int_\Omega q^{(\frac{1}{2}+is)\langle \gamma x, \gamma \omega \rangle} \mu_{s, \phi}(\mathrm{d} \gamma\omega) \\
    &=&\int_\Omega q^{(\frac{1}{2}+is)\langle x, \omega \rangle} q^{(\frac{1}{2}+is)\langle \gamma o, \gamma \omega \rangle} \mu_{s, \phi}(\mathrm{d} \gamma \omega),
  \end{eqnarray*}
    where we used the horocycle identity in Lemma~\ref{lem:brackets_propI}~\ref{eq:horocycle_identity}.
\end{proof}

\section{Patterson-Sullivan distributions} \label{sect:PS_Dist}
We present two equivalent descriptions of Patterson–Sullivan distributions of finite graphs: the \emph{classical one} based on Helgason boundary values and the (weighted) Radon transform, and a more \emph{dynamical} perspective relying on the quantum-classical correspondence, that is, an isomorphism between eigenspaces of Laplacians and eigenspaces of transfer (or Koopman) operators for finite graphs \cite[§4]{AFH23}.

\subsection{Description in terms of (weighted) Radon transform}
\label{sect:descriptionPS_classical}
Motivated by the original definition of Patterson–Sullivan distributions on compact locally symmetric spaces \cite{AZ07, AZ12, HHS12}, we begin by introducing the weighted Radon transform as a first step towards defining their analogue on trees.

\subsubsection{Weighted Radon transforms}
\label{sect:Radontransform}

{Recall the definition of $r \in K$ from Section~\ref{sect:homogeneous_space}.}
\begin{definition}[Intermediate values]
  For $s , s' \in \mathbb{C}$ we define
  \begin{equation*}
    d_{s , s'} \colon G/M \rightarrow \mathbb{C}, \quad d_{s, s'}(gM) \coloneqq q^{(\frac{1}{2}+is)\langle go, g \omega_+ \rangle}q^{(\frac{1}{2}+is')\langle go,g \omega_- \rangle} = q^{(\frac{1}{2}+is)H(g)}q^{(\frac{1}{2}+is')H(gr)}.
  \end{equation*}
\end{definition}

\begin{lem} \label{lem:d_Radon}
   Let $g, g'\in G$ and $n \in \mathbb{Z}$. Then
   \begin{enumerate}[label=(\roman*)]
   \item $d_{s , s'}(g' gM) = q^{(\frac{1}{2}+is)\langle g' o, g' g \omega_+ \rangle}q^{(\frac{1}{2}+is')\langle g' o, g' g \omega_- \rangle}d_{s , s'}(gM),$
   \item $d_{s , s'}(g \tau^n M) = q^{n(\frac{1}{2}+is)} q^{-n(\frac{1}{2}+is')}d_{s ,s'}(gM)$.
   \end{enumerate}
\end{lem}
\begin{proof}
  The first part follows from Lemma~\ref{lem:brackets_propII}. For the second part, recall that $r \tau^n r^{-1} = \tau^{- n},\ r \in K,$ and note that
  \begin{align*}
    d_{s , s'}(g \tau^n)
    = q^{(\frac{1}{2}+is)H(g \tau^n)}q^{(\frac{1}{2}+is')H(g \tau^n r)}
    &= q^{(\frac{1}{2}+is)(H(g) + n)}q^{(\frac{1}{2}+is')(H(gk) + H(r^{-1} \tau^n r))} \\
    &= d_{s , s'}(gM)q^{n(\frac{1}{2}+is)} q^{-n(\frac{1}{2}+is')}.\qedhere
  \end{align*}
\end{proof}

\begin{definition}[Radon transform] \label{def:Radon_transform}
  For a function $f \colon G / M \rightarrow \mathbb{C}$ we define the \emph{weighted Radon transform} $\mathcal{R}_{s , s'}$ on $G/M$ by
  \begin{equation} \label{eq:Radon_transform}
    (\mathcal{R}_{s,s'}f)(g) \coloneqq \sum_{
      j \in \mathbb{Z}
    }f(g \tau^j)d_{s,- \overline{s}'}(g \tau^j),
  \end{equation}
  whenever the series converges.
\end{definition}
This is indeed right $M$-invariant since $M$ normalizes $\langle \tau \rangle$.
Note also that, if it exists, $\mathcal{R}_{s,s'}f$ is right $\langle \tau \rangle$-invariant and hence a function on $G/M \langle \tau  \rangle$.

\begin{lem} \label{lem:Radon_smoothOmega}
  If $f \in C_c^{\mathrm{lc}}(G/M)$, then $\mathcal{R}_{s, s'}f \in C_c^{\mathrm{lc}}(G/M \langle \tau \rangle) \cong C_c^{\mathrm{lc}}((\Omega \times \Omega) \backslash \mathrm{diag}(\Omega))$.
\end{lem}

\begin{proof}
  If $\mathrm{pr} \colon G/M \rightarrow G/M \langle \tau \rangle$ denotes the canonical projection, we have
  \begin{equation*}
    \forall gM \langle \tau \rangle \notin  \mathrm{pr}(\mathrm{supp}f),\ j \in \mathbb{Z} \colon \quad f^{\langle j\rangle}(gM) \coloneqq f(g \tau^j M) = 0,
  \end{equation*}
  which implies that $\mathcal{R}_{s,s'}f(gM) =0.$
  The isomorphism follows from Proposition~\ref{prop:open_cells}.
\end{proof}
Note that by duality, we obtain the operator
\begin{equation}\label{eq:dualRadon}
  \mathcal{R}'_{s,s'}: \mathcal{D}'((\Omega \times \Omega)\backslash \mathrm{diag}(\Omega)) \rightarrow \mathcal{D}'(G/M).
\end{equation}

\begin{rem} \label{rem:Radon_tau}
  By Lemma~\ref{lem:d_Radon} we have for $s,s' \in \C,\, n \in \mathbb{Z} $ and $f \in C_c^{\mathrm{lc}}(G/M)$ that
  $$(\mathcal{R}_{s,s'}f^{\langle n\rangle})=q^{n(\frac{1}{2}+is)}q^{-n(\frac{1}{2}- i \overline{s}')}\mathcal{R}_{s,s'}(f).$$
  Moreover, in the coordinates $(\omega_1, \omega_2) \in (\Omega \times \Omega) \backslash \mathrm{diag}(\Omega)$, we have for $f \in C_c^{\mathrm{l c}}(\mathfrak{P})$
  \begin{equation*}
    (\mathcal{R}_{s,s'}f)(\omega_1 , \omega_2) = \sum_{
      x \in(\omega_1 , \omega_2)
    }f(\omega_1 , \omega_2, x)p_s(x, \omega_1)p_{- \overline{s}'}(x, \omega_2),
  \end{equation*}
  with the Poisson kernel from Definition~\ref{def:Poisson_transform}. This follows, using Proposition~\ref{prop:psi_hom}, from
  \begin{equation*}
    (\mathcal{R}_{s, s'}f)(g) = \sum_{j \in \mathbb{Z}} f(g \tau^j \omega_- , g \tau^j \omega_+ , g \tau^jo)p_s(g \tau^j o, g \tau^j \omega_+)p_{- \overline{s}'}(g \tau^j o, g \tau^j \omega_-),
  \end{equation*}
  where $g \tau^jo$ runs over all vertices on the geodesic $(g\omega_- , g \omega_+ )$ and $\tau^j \omega_{\pm } = \omega_{\pm }$.
\end{rem}

\begin{prop} \label{prop:radon_inv}
  Let $s,s' \in \C$ and $f \in C_c^{\mathrm{lc}}(G/M)$. For $\gamma \in G$, set $f_\gamma(gM)\coloneqq f(\gamma^{-1}gM).$
  Then for $\omega_1, \omega_2 \in \Omega$, the following equivariance property holds:
  $$(\mathcal{R}_{s,s'}f_\gamma)(\omega_1, \omega_2)=q^{(\frac{1}{2}+is)\langle \gamma o, \omega_1 \rangle} q^{(\frac{1}{2}- i \overline{s}')\langle \gamma o, \omega_2 \rangle} (\mathcal{R}_{s,s'}f)(\gamma^{-1} \omega_1, \gamma^{-1}\omega_2).$$
\end{prop}

\begin{proof}
  It suffices to prove it for $(\omega_1, \omega_2)=(g \omega_-, g \omega_+) \in (\Omega \times \Omega)\backslash \mathrm{diag}(\Omega)$ due the previous remark. Note that $gM\langle \tau \rangle$ is determined uniquely by $(\omega, \omega')$. Then we calculate
  \begin{align*}
    (\mathcal{R}_{s,s'}f_\gamma)(gM \langle \tau \rangle)
      &= \sum_{n \in \Z} d_{s, - \overline{s}'}(g \tau^nM)f(\gamma^{-1} g \tau^n M) \\
      &= \sum_{n \in \Z} d_{s,- \overline{s}'}(\gamma^{-1} g \tau^nM)f(\gamma^{-1} g \tau^n M) q^{-(\frac{1}{2}+is)\langle \gamma^{-1}o, \gamma^{-1} g \omega_+ \rangle}q^{-(\frac{1}{2}- i \overline{s}')\langle \gamma^{-1} o, \gamma^{-1} g \omega_- \rangle} \\
      &= \sum_{n \in \Z} d_{s,- \overline{s}'}(\gamma^{-1} g \tau^nM)f(\gamma^{-1} g \tau^n M) q^{(\frac{1}{2}+is)\langle \gamma o,  g \omega_+ \rangle}q^{(\frac{1}{2}- i \overline{s}')\langle \gamma o,  g \omega_- \rangle} \\
      &= q^{(\frac{1}{2}+is)\langle \gamma o,  g \omega_+ \rangle}q^{(\frac{1}{2} - i \overline{s}')\langle \gamma o,  g \omega_- \rangle} (\mathcal{R}_{s,s'}f)(\gamma^{-1} gM \langle \tau \rangle),
  \end{align*}
  where we first used Lemma~\ref{lem:d_Radon} and then Lemma~\ref{lem:brackets_propI}.
\end{proof}

\subsubsection{Classical Patterson-Sullivan distributions}
\label{sect:PS_classical}
The Patterson–Sullivan distributions on $G/M$ are defined as pairings between $\chi(s)$-boundary values of eigenfunctions and test functions via a weighted Radon transform, as follows.

\begin{definition}[Patterson-Sullivan distribution on $G/M$]
\label{def:PS_GM}
Fix $s,s' \in \C$ and consider eigenfunctions $\phi \in \mathcal{E}_{\chi(s)}(\Laplace; \mathrm{Maps}(\mathfrak{X}, \mathbb{C})), \,\phi' \in \mathcal{E}_{\chi(s')}(\Laplace; \mathrm{Maps}(\mathfrak{X}, \mathbb{C}))$. Let $\mu_{s,\phi},\, \mu_{s',\phi'} \in \mathcal{D}'(\Omega)$ denote their respective boundary values. The \emph{Patterson-Sullivan distribution} $\mathrm{P S}_{\phi,\phi'} \in \mathcal{D}'(G/M)$ associated to $\phi$ and $\phi'$ is defined by
\begin{equation*} \label{eq:PS_GM}
  \langle f, \mathrm{P S}_{\phi,\phi'} \rangle_{G/M} \coloneqq
  \int_{(\Omega \times \Omega)\backslash \mathrm{diag}(\Omega)} \mathcal{R}_{s,s'}(f)(\omega,\omega')
  \; \mu_{s,\phi}(\mathrm{d}\omega) \overline{\mu_{s',\phi'}}(\mathrm{d}\omega'),
\end{equation*}
    where $f\in C_c^{\mathrm{lc}}(G/M)$ is a test function.
  \end{definition}

  Note that this makes sense since $(\Omega \times \Omega)\backslash \mathrm{diag}(\Omega)$ is open in $\Omega \times \Omega$ (see Proposition~\ref{prop:open_cells}), hence the distribution
  $\mu_{s,\phi}(\mathrm{d}\omega) \otimes \overline{\mu_{s',\phi'}}(\mathrm{d}\omega')$ on $\Omega \times \Omega$ can be restricted to $(\Omega \times \Omega)\backslash \mathrm{diag}(\Omega)$. Since by Lemma~\ref{lem:Radon_smoothOmega} the Radon transform $\mathcal{R}_{s,s'}(f)$ is compactly supported in $(\Omega \times \Omega)\backslash \mathrm{diag}(\Omega)$, we obtain
  \begin{equation*} \label{eq:PS_B}
    \langle f, \mathrm{P S}_{\phi,\phi'} \rangle_{G/M} \coloneqq
    \int_{\Omega \times \Omega} \mathcal{R}_{s,s'}(f)(\omega,\omega')
    \; \mu_{s,\phi}(\mathrm{d}\omega) \overline{\mu_{s',\phi'}}(\mathrm{d}\omega').
  \end{equation*}

      Next, we want to define the Patterson-Sullivan distribution on the quotient $\Gamma \backslash G/M$. For this we need the following result.

      \begin{prop}
        Suppose that $\phi \in \mathcal{E}_{\chi(s)}(\Laplace; \mathrm{Maps}(\mathfrak{X}, \mathbb{C})),\,\phi' \in \mathcal{E}_{\chi(s')}(\Laplace; \mathrm{Maps}(\mathfrak{X}, \mathbb{C}))$ are $\Gamma$-invariant eigenfunctions with spectral parameters $s,s' \in \C$ respectively. Then the Patterson-Sullivan distribution $\mathrm{P S}_{\phi,\phi'} \in \mathcal{D}'(G/M)$ is $\Gamma$-invariant.
      \end{prop}

      \begin{proof}
    For $f\in C_c^{\mathrm{lc}}(G/M)$, we have
    \begin{eqnarray*}
      \langle f_\gamma, \mathrm{P S}_{\phi,\phi'} \rangle_{G/M}
      &=& \int_{(\Omega \times \Omega)\backslash \mathrm{diag}(\Omega)}
            (\mathcal{R}_{s,s'}f_\gamma)(\omega, \omega')\;  \mu_{s, \phi}(\mathrm{d}\omega) \overline{\mu_{s',\phi'}}(\mathrm{d}\omega ') \\
        &=& \int_{(\Omega \times \Omega)\backslash \mathrm{diag}(\Omega)}
        (\mathcal{R}_{s,s'}f_\gamma)(\gamma. (\omega, \omega'))\cdot\\
          && \qquad \qquad \qquad  q^{-(\frac{1}{2}+is)\langle \gamma o, \gamma \omega \rangle}q^{-(\frac{1}{2} - i \overline{s}')\langle \gamma o, \gamma \omega' \rangle}
             \; \mu_{s, \phi}(\mathrm{d}\omega) \overline{\mu_{s',\phi'}}(\mathrm{d}\omega ') \\
        &=& \int_{(\Omega \times \Omega)\backslash \mathrm{diag}(\Omega)}
            (\mathcal{R}_{s,s'}f)(\omega, \omega')\;  \mu_{s, \phi}(\mathrm{d}\omega) \overline{\mu_{s',\phi'}}(\mathrm{d}\omega') \\
      &=&  \langle f, \mathrm{P S}_{\phi,\phi'} \rangle_{G/M},
    \end{eqnarray*}
    where we first used the $\Gamma$-invariance of the boundary measures from Lemma~\ref{lem:boundaryvalues_inv}, recalling that $\overline{\mu_{s', \phi '}} = \mu_{- \overline{s}',\overline{\phi '}}$, and then applied Proposition~\ref{prop:radon_inv}.
  \end{proof}

  \begin{rem}
    By Remark~\ref{rem:Radon_tau}, we see that for $n \in \Z$ and $\Gamma$-invariant eigenfunctions $\phi \in \mathcal{E}_{\chi(s)}(\Laplace; \mathrm{Maps}(\mathfrak{X}, \mathbb{C})),$ $\phi' \in \mathcal{E}_{\chi(s')}(\Laplace; \mathrm{Maps}(\mathfrak{X}, \mathbb{C}))$:
    $$\langle f^{\langle n \rangle}, \mathrm{P S}_{\phi,\phi'} \rangle_{G/M} = q^{n(\frac{1}{2}+is)} q^{-n(\frac{1}{2}- i \overline{s}')} \langle f, \mathrm{P S}_{\phi,\phi'} \rangle_{G/M}.$$
    This means that the associated Patterson-Sullivan distribution is invariant under the right-translation by $\langle \tau \rangle$.
\end{rem}

\begin{definition}[Patterson-Sullivan distribution on $\Gamma \backslash G/M$]\label{def:PS_GammaGM}
  \label{def:PS_quotient}
  Let $s,s' \in \C$ and $\phi \in \mathcal{E}_{\chi(s)}(\Laplace; \mathrm{Maps}(\mathfrak{X}, \mathbb{C}))^\Gamma,\,\phi' \in \mathcal{E}_{\chi(s')}(\Laplace; \mathrm{Maps}(\mathfrak{X}, \mathbb{C}))^\Gamma$ be $\Gamma$-invariant eigenfunctions. Consider $\Xi \in C_c^{\mathrm{lc}}(G/M)$ a lcfd-function.
  Then the $\Gamma$-invariant distribution $\mathrm{P S}_{\phi,\phi'}$ on $G/M$ descends to the quotient $\Gamma\backslash G/M$ via
  $$\langle f, \mathrm{P S}^\Gamma_{\phi,\phi'} \rangle_{\Gamma\backslash G/M} 
  \coloneqq \langle \Xi(f \circ \pi_\Gamma ), \mathrm{P S}_{\phi,\phi'}\rangle_{G/M}
  \;\;\;\; \forall f \in C^{\mathrm{lc}}(\Gamma\backslash G/M),$$
  where $\pi_{\Gamma}: G/M \rightarrow \Gamma \backslash G/ M$ is the canonical projection.
\end{definition}

In particular, Proposition~\ref{prop:cutoff_indep} ensures that $\mathrm{P S}^\Gamma_{\phi,\phi'}$ is independent of the choice of $\Xi\in C_c^{\mathrm{lc}}(G/M)$.

\subsection{Description in terms of resonant and coresonant states}
\label{sect:descriptionPS_dynamical}
Next, we connect the classical Patterson-Sullivan distributions to the shift dynamics on the space of functions on $\mathfrak{P}_\Gamma$.
For this, we first recall the definition of resonant and coresonant states on regular graphs from \cite{AFH23Pairing}.

\subsubsection{Resonances on regular graphs}
\label{sect:resonances}
We first introduce the shift dynamics on the set $\mathfrak{P}^+$ of chains (introduced in Section~\ref{sect:harmonic_analysis_graphs}), which we call the \textit{shift space}.
By reversing the orientation, we obtain the \textit{opposite shift space} $\mathfrak{P}^-$, which is given by the set of infinite, non-backtracking sequences of concatenated edges $\mathbf{p}_{-}\coloneqq(\dots, \vec{e}_2, \vec{e}_1).$ We equip $\mathfrak{P}^\pm$ with the (discrete) topology generated by the sets of all chains of edges $\mathbf{p}_{+}\coloneqq \mathbf{p}=(\vec{e}_1, \vec{e}_2, \dots)$ respectively $\mathbf{p}_{-}$ such that $(\vec{e}_1, \vec{e}_2, \dots, \vec{e}_n)$ is equal to some fixed tuple of edges and $n\in \N$, \cite[§5]{BHW23}.
In particular, since chains on $\mathfrak{X}$ can be interpreted as geodesic rays, $\mathfrak{P}^\pm$ is one possible graph-analogue of the sphere bundle of a Riemannian manifold, i.e., $\mathfrak{P}^\pm \cong \mathfrak{X} \times \Omega = S\mathfrak{X}$.

\begin{definition}[Ruelle transfer operator and its dual, {\cite[§3]{BHW23} and \cite[§1.2]{AFH23}}] \label{def:Ruelle_transfer_op}
  On the space $\mathrm{Maps}(\mathfrak{P}^+ , \mathbb{C})$ of complex valued functions on $\mathfrak{P}^+$ of a regular graph, we define the \emph{Ruelle transfer operator}
  $\mathcal{L}_+ : \mathrm{Maps}(\mathfrak{P}^+ , \mathbb{C}) \rightarrow \mathrm{Maps}(\mathfrak{P}^+ , \mathbb{C})$ as
  \begin{equation*}
    \mathcal{L}_+ f(\mathbf{p}_{+}) \coloneqq \sum_{
      \vec{e}_0 \colon \tau(\vec{e}_0) = \iota(\vec{e}_1)}f(\mathbf{p}'_{+}),
    \end{equation*}
    where $\mathbf{p}'_{+}=(\vec{e}_0 , \vec{e}_1 , \ldots) \in \mathfrak{P}^+ $,
    and its dual operator, known as the \textit{Koopman operator},
    \begin{equation*}
      \mathcal{L} '_+   \colon \mathcal{D}'(\mathfrak{P}^+) \rightarrow \mathcal{D}'(\mathfrak{P}^+),
    \end{equation*}
    where $\mathcal{D}'(\mathfrak{P}^+)$ denotes the algebraic dual of space $C_c^{\mathrm{l c}}(\mathfrak{P}^+)$ of locally constant functions with compact support (note that $\mathcal{L}_+  $ leaves $C_c^{\mathrm{l c}}(\mathfrak{P}^+)$ invariant).
  \end{definition}

  Reversing the orientation of the edge paths, we get similar operators $\mathcal{L}_-$ and $\mathcal{L}_- '$ on $\mathrm{Maps}(\mathfrak{P}^- , \mathbb{C})$ respectively $\mathcal{D}'(\mathfrak{P}^-)$.

  \begin{definition}[(Co)-resonant states, eigenspaces and resonances, {\cite[Def.~1.3]{AFH23Pairing}}] \label{def:coresonant_states}
    The \emph{(co)-resonant states} are the eigendistributions of $\mathcal{L}_+'$ (resp.\@ $\mathcal{L}_-' $), i.e., non-zero elements of the \emph{eigenspaces},
    \begin{equation*}
      \mathcal{E}_s(\mathcal{L}_{\pm}'; \mathcal{D}'(\mathfrak{P}^{\pm} )) \coloneqq \left\{ u \in \mathcal{D}'(\mathfrak{P}^{\pm}) \mid  \mathcal{L}_{\pm}' u = q^{\frac{1}{2} + is } u \right\}, \quad \forall s \in \C.
    \end{equation*}
    If there exists a non-zero element $u \in \mathcal{E}_s(\mathcal{L}_{\pm}'; \mathcal{D}'(\mathfrak{P}^+ ))$ for $s\in \C$, we say that $s$ is a \emph{(classical) resonance}.
  \end{definition}

  Consider the base projections \cite[Def.~1.4]{AFH23Pairing}, which relate functions on chains to functions on edges and vertices
  \begin{eqnarray*}
    \pi_+ \coloneqq \iota \circ \pi^{\mathfrak{E}}_+ \colon \mathfrak{P}^+ \rightarrow& \mathfrak{X},& (\vec{e}_1, \vec{e}_2, \dots) \mapsto \iota(\vec{e}_1) \\
    \pi_- \coloneqq \tau \circ \pi^{\mathfrak{E}}_- \colon \mathfrak{P}^- \rightarrow& \mathfrak{X},& (\dots, \vec{e}_2, \vec{e}_1) \mapsto \tau(\vec{e}_1),
\end{eqnarray*}
where $\pi^{\mathfrak{E}}_\pm:  \mathfrak{P}^\pm \rightarrow \mathfrak{G}$ are given by $\pi^{\mathfrak{E}}_+(\vec{e}_1, \vec{e}_2, \dots)\coloneqq \vec{e}_1$ and $\pi^{\mathfrak{E}}_-(\dots, \vec{e}_2, \vec{e}_1)\coloneqq \vec{e}_1$, respectively.
The \emph{end point projections} \cite[Def.~2.1]{AFH23Pairing} are given by
\begin{eqnarray*}
    B_+ \colon \mathfrak{P}^+ \rightarrow& \Omega,& (\vec{e_1}, \vec{e_2}, \dots) \mapsto [ (\vec{e_1}, \vec{e_2}, \dots)] \\
    B_- \colon \mathfrak{P}^- \rightarrow& \Omega,& (\dots, \vec{e}_2, \vec{e}_1) \mapsto [ (\dots, \vec{e}_2, \vec{e}_1)].
\end{eqnarray*}
The above projections induce well-defined pushforward maps
$\pi_{\pm , *} \colon \mathcal{D}'(\mathfrak{P}^{\pm}) \rightarrow \mathrm{Maps}_c(\mathfrak{X}, \mathbb{C})'$ and pullbacks $B_{\pm}^{\#} \colon \mathcal{D}'(\Omega) \rightarrow \mathcal{D}'(\mathfrak{P}^{\pm })$,
see \cite[Prop.\@~1.6 and Def.\@~2.1]{AFH23Pairing} for details.
These maps are connected to the Poisson transform $\mathcal{P}_s$ (see Definition~\ref{def:Poisson_transform}) in the following way:

\begin{prop}[{see \cite[Props.~2.7, 2.11, 3.1 and 3.2]{AFH23}}]
\label{prop:classical_quantum_correspondence}
Let $\mathfrak{G}_\Gamma$ denote a connected, finite regular graph and $\mathfrak{G} =(\mathfrak{X}, \mathfrak{E})$ its universal cover, a homogeneous tree. We have the isomorphisms
\begin{equation*}
    p_{s}B_{\pm}^{\#} \colon \mathcal{D}'(\Omega) \stackrel{\sim}{\longrightarrow} \mathcal{E}_s(\mathcal{L}_{\pm}'; \mathcal{D}'(\mathfrak{P}^{\pm})), \qquad p_{s}B_{\pm}^{\#}(\mu)(\varphi) \coloneqq B_{\pm}^{\#}(\mu)(p_{s}\varphi)
  \end{equation*}
  and
  \begin{equation}\label{eq:Poisson_psBs}
    \mathcal{P}_s = \pi_{\pm , *} \circ p_{s}B_{\pm}^{\#}.
  \end{equation}
  Moreover, the canonical projection $\pi \colon \mathfrak{X} \rightarrow \mathfrak{X}_\Gamma$ and its analogues $\pi_{\mathfrak{P}^{\pm}} \colon \mathfrak{P}^{\pm} \rightarrow \mathfrak{P}^{\pm}_\Gamma$ for the path spaces induce isomorphisms
  \begin{eqnarray*}
    \pi^{*} \colon& \mathcal{E}_{\chi(s)}(\Laplace_{\Gamma}; \mathrm{Maps}(\mathfrak{X}_\Gamma, \mathbb{C})) &\stackrel{\sim}{\longrightarrow} \;\; \mathcal{E}_{\chi(s)}(\Laplace ; \mathrm{Maps}(\mathfrak{X}, \mathbb{C}))^{\Gamma},\;\; \pi^{*}(f)(x) = f(\pi(x)) ,\\
    \pi_{\mathfrak{P}^{\pm}}^{\#} \colon& \mathcal{E}_{s}(\mathcal{L}_{\Gamma, \pm}';\mathcal{D}'( \mathfrak{P}^{\pm}_\Gamma)) \;\;\qquad &\stackrel{\sim}{\longrightarrow} \;\; \mathcal{E}_s(\mathcal{L}_{\pm}';\mathcal{D}'(\mathfrak{P}^{\pm}))^{\Gamma},
  \end{eqnarray*}
  where $\Laplace_\Gamma$ and $\mathcal{L}'_{\Gamma , \pm}$ denote the Laplacian and transfer operators on $\mathfrak{X}_\Gamma$ and $ \mathfrak{P}^{\pm}_\Gamma\coloneqq\Gamma \backslash \mathfrak{P}^{\pm}$ respectively.
\end{prop}
In addition, by \cite[Lem.~2.22]{AFH23}, we have that the map $ p_{s}B_{\pm}^{\#}$ intertwines the representation $\pi_s$ defined in \eqref{eq:rep_pis} with the left regular representation on $\mathcal{D}'(\mathfrak{P}^{\pm})$.

\subsubsection{Dynamical Patterson-Sullivan distributions}
\label{sect:PS_dynamical}
Let $\mathfrak{G}$ and $\mathfrak{G}_\Gamma$ be as in Proposition~\ref{prop:classical_quantum_correspondence}. Consider a non-exceptional spectral parameter $s \in \mathbb{C}$, i.e., $\chi(s) \notin \{\pm 1\}$. Since the Poisson transform $\mathcal{P}_s$ is invertible by Theorem~\ref{thm:Poissontransform_delta}, the pushforwards $\pi_{\pm , *}$ induce isomorphisms between $\mathcal{E}_s(\mathcal{L}_{\pm}';\mathcal{D}'(\mathfrak{P}^{\pm}))$ and $\mathcal{E}_{\chi(s)}(\Laplace ; \mathrm{Maps}(\mathfrak{X},\mathbb{C}))$, \cite[Thm.~4.4]{AFH23}. This yields the regular \emph{quantum–classical correspondence} for graphs.
Let us consider the inverses of these isomorphisms on the quotients, i.e.,
\begin{equation*}
  \mathbf{I}_{\pm}(s) \coloneqq(\pi_{\mathfrak{P}^{\pm}}^{\#})^{-1} \circ(\pi_{\pm , *})^{-1}\colon \mathcal{E}_{\chi(s)}(\Laplace ; \mathrm{Maps}(\mathfrak{X}, \mathbb{C}))^{\Gamma} \stackrel{\cong}{\longrightarrow} \mathcal{E}_s(\mathcal{L}_{\Gamma , \pm}'; \mathcal{D}'( \mathfrak{P}^{\pm}_\Gamma)).
\end{equation*}

For two distributions $u_{\pm} \in \mathcal{D}'( \mathfrak{P}^{\pm}_\Gamma)$ we define their tensor product $u_+ \otimes u_- \in \mathcal{D}'(\mathfrak{P}^+_\Gamma \times \mathfrak{P}^-_\Gamma)$
\begin{equation*}
  (u_+ \otimes u_-)(\varphi) \coloneqq \langle u_+ , \Gamma \mathbf{p}_+ \mapsto \langle u_- , \varphi(\cdot , \Gamma \mathbf{p}_+) \rangle \rangle,
\end{equation*}
where $\Gamma \mathbf{p}_+ \in \mathfrak{P}^{+}_\Gamma $. Moreover, considering functions $\varphi \in C(\mathfrak{P}_\Gamma )$ as functions on $\mathfrak{P}^{+ }_\Gamma  \times \mathfrak{P}^{- }_\Gamma $ by extending them by zero we may consider $u_+ \otimes u_-$ as an element of $\mathcal{D}'(\mathfrak{P}_\Gamma ) \cong \mathcal{D}'(\Gamma \backslash G /M)$, by recalling the identification from Proposition~\ref{prop:psi_hom}.
The connection to the Patterson-Sullivan distributions is as follows:

\begin{thm}\label{thm:PS_modern}
  Let $s,s' \in \mathbb{C}$ be two non-exceptional spectral parameters and $\phi \in \mathcal{E}_{\chi(s)}(\Laplace;$ $ \mathrm{Maps}(\mathfrak{X}, \mathbb{C}))^\Gamma$, $\phi' \in \mathcal{E}_{\chi(s')}(\Laplace; \mathrm{Maps}(\mathfrak{X}, \mathbb{C}))^\Gamma$ be $\Gamma$-invariant eigenfunctions. Then the Patterson-Sullivan distribution $\mathrm{PS}_{\phi,\phi'}^\Gamma \in \mathcal{D}'(\mathfrak{P}_\Gamma)$ can be written as
  \begin{equation*}
    \mathrm{PS}^{\Gamma}_{\phi , \phi '} = \mathbf{I}_+(s)(\phi) \otimes \mathbf{I}_-(- \overline{s}')(\overline{\phi '}).
  \end{equation*}
\end{thm}
\begin{proof}
  We first claim that it is enough to prove that
  \begin{equation}\label{eq:PS_modern_notinvariant}
    \mathrm{PS}_{\phi , \phi '} = (\pi_{+ ,*})^{-1}(\phi) \otimes(\pi_{- , *})^{-1}(\overline{\phi '}).
  \end{equation}
  To see this, we recall from \cite[Eq.\@ (15)]{AFH23}, that the inverse of $\pi_{\mathfrak{P}^{\pm}}^{\#}$ is given by
  \begin{equation*}
    \langle(\pi_{\mathfrak{P}^{\pm}}^{\#})^{-1}(U), F \rangle = \langle U,(F \circ \pi_{\mathfrak{P}^{\pm}})\Xi_{\mathfrak{P}^{\pm}} \rangle,
  \end{equation*}
  where $U \in \mathcal{D}'(\mathfrak{P}^{\pm})^\Gamma,\ F \in C( \mathfrak{P}^{\pm}_\Gamma)$ and $\Xi_{\mathfrak{P}^{\pm}} \in C_c^{\mathrm{l c}}(\mathfrak{P}^{\pm})$ is a lcfd-function on $\mathfrak{P}^{\pm}$, i.e., a function such that $\sum_{
    \gamma \in \Gamma
  }\Xi_{\mathfrak{P}^{\pm}}(\gamma . \mathbf{p}_{\pm}) = 1$ for every $\mathbf{p}_{\pm} \in \mathfrak{P}^{\pm}$. We may choose these cutoffs compatible in the sense that $\mathbbm{1}_{\mathfrak{P}}\Xi_{\mathfrak{P}^{+}}\Xi_{\mathfrak{P}^{-}}$ defines a lcfd-function $\Xi_{\mathfrak{P}}$ on $\mathfrak{P}$ (for example we may take a fundamental domain in $\mathfrak{X}$). Then we obtain that
  \begin{equation*}
    \langle(\pi_{\mathfrak{P}^+}^{\#})^{-1}((\pi_{+ , *})^{-1}(\phi)) \otimes (\pi_{\mathfrak{P}^- }^{\#})^{-1}((\pi_{-  , *})^{-1}(\overline{\phi '})), F \rangle = \langle(\pi_{+ , *})^{-1}(\phi) \otimes(\pi_{- , *})^{-1}(\overline{\phi '}), (F \circ \pi_{\mathfrak{P}})\Xi_{\mathfrak{P}} \rangle
  \end{equation*}
  and hence, Equation~\eqref{eq:PS_modern_notinvariant} would prove the proposition by Definition~\ref{def:PS_GammaGM}. Let us now prove Equation~\eqref{eq:PS_modern_notinvariant}. By Remark~\ref{rem:Radon_tau} we infer for $f \in C^{\mathrm{lc}}_c({\mathfrak{P}})$
  \begin{align*}
    \mathrm{PS}_{\phi , \phi '}(f) &= \int_{(\Omega \times \Omega)\backslash \mathrm{diag}(\Omega)} \mathcal{R}_{s,s'}(f)(\omega,\omega')
    \mu_{s,\phi}(\mathrm{d}\omega) \overline{\mu_{s',\phi'}}(\mathrm{d}\omega') \\
    &=
    \int_{(\Omega \times \Omega)\backslash \mathrm{diag}(\Omega)} \sum_{
      x \in(\omega , \omega')
      }f(\omega , \omega', x)p_s(x, \omega)p_{- \overline{s}'}(x, \omega')
      \; \mu_{s,\phi}(\mathrm{d}\omega) \overline{\mu_{s',\phi'}}(\mathrm{d}\omega').
  \end{align*}
  For the right hand side of Equation~\eqref{eq:PS_modern_notinvariant}, we use Equation~\eqref{eq:Poisson_psBs} to obtain
  \begin{equation*}
    ((\pi_{+ ,*})^{-1}(\phi) \otimes(\pi_{- , *})^{-1}(\overline{\phi '}))(f) = (p_s B_+^{\#}(\mu_{s, \phi}) \otimes p_{- \overline{s}'}B_-^{\#}(\mu_{- \overline{s}', \overline{\phi '}}))(f).
  \end{equation*}
  Using the explicit definition of $B_{\pm}^{\#}$ from \cite[Eq.\@ (9)]{AFH23} we infer that
  \begin{align*}
    p_{- \overline{s}'} B_{-}^{\#}(\mu_{- \overline{s}', \overline{\phi'}})(f(\cdot , {\mathbf{p}_+}))
    &= B_{- }^{\#}(\mu_{- \overline{s}', \overline{\phi '}})(p_{- \overline{s}'}f(\cdot , {\mathbf{p}_+})) \\
    & = \mu_{- \overline{s}',\overline{\phi '}}(B_{- , \#}(p_{- \overline{s}'} f(\cdot , {\mathbf{p}_+}))) \\
    & = \int_{\Omega}\sum_{ x \in \mathfrak{X}}p_{- \overline{s}'}(x, \omega')f((\omega' , x), {\mathbf{p}_+}) \; \mathrm{d}\mu_{- \overline{s}', \overline{\phi '}}(\omega')
  \end{align*}
  for every fixed $\mathbf{p}_-  \in \mathfrak{P}^-$. A similar calculation for $p_s B_+^{\#}(\mu_{s,\phi})$ allows us to write
  \begin{eqnarray*}
      &&((\pi_{+ ,*})^{-1}(\phi) \otimes(\pi_{- , *})^{-1}(\overline{\phi '}))(f)\\ &=&
    \int_{(\Omega \times \Omega) \backslash \mathrm{diag}(\Omega)} \sum_{
      x \in \mathfrak{X}
    }p_s(x, \omega)p_{- \overline{s}'}(x, \omega')f(\omega, \omega', x) \; \mu_{s, \phi}(\mathrm{d}\omega) \mu_{- \overline{s}', \overline{\phi'}}(\mathrm{d}\omega').
  \end{eqnarray*}
  Note that we only have to sum over $\mathfrak{X}$ once since $f$ is supported on $\mathfrak{P}$ so that the starting points have to be the same. By the same argument, only the pairs $(\omega, \omega')$ off the diagonal and only those $x \in \mathfrak{X}$ which lie on the geodesic $(\omega, \omega')$ contribute non-zero terms in the latter equation. The proposition then follows from \eqref{eq:complex_conj}.
\end{proof}

\begin{rem}[Connection with the pairing formula]\label{rem:relation_pairing_formula}
  If we evaluate the tensor product $\mathbf{I}_+(s)(\phi) \otimes \mathbf{I}_-(- \overline{s}')(\overline{\phi '})$ at the characteristic function $\mathbbm{1}_{\mathfrak{P}_\Gamma }$ we get the \emph{geodesic pairing} from \cite[Def.\@ 1.8]{AFH23Pairing}. By \cite[Lem.~1.9 and Thm.~4.16]{AFH23Pairing} we have $\mathrm{PS}_{\phi ,\phi '}^\Gamma (\mathbbm{1}_{\mathfrak{P}_\Gamma }) = 0$ if $s \neq - \overline{s}'$ and
  \begin{equation*}
    (q^{1 + 2is} - 1)\mathrm{PS}_{\phi , \phi '}^\Gamma (\mathbbm{1}_{\mathfrak{P}_\Gamma }) = (q^{1 + 2 is}- q)\sum_{
      \Gamma x \in \mathfrak{X}_\Gamma
    }\phi(\Gamma x)\overline{\phi '}(\Gamma x)
    \end{equation*}
    if $s = - \overline{s}'$, see \cite[Def.~1.7 and Lem.~3.2]{AFH23Pairing}.
  \end{rem}

\section{Invariant Ruelle distributions}
\label{sect:InvariantRuelleDist}
To define the invariant Ruelle distributions on finite graphs via the transfer operator of the geodesic flow, we first need to introduce the notion of the trace of a finite-rank operator.

\begin{lem} Let $\mathfrak{P}^\pm_\Gamma$ be the shift space of the quotient. Then,  we have an isomorphism
  \begin{gather*}
    \mathcal{D}'(\mathfrak{P}^+_\Gamma  \times\mathfrak{P}^-_\Gamma ) \cong \mathrm{Hom}(C^{\mathrm{l c}}(\mathfrak{P}^+_\Gamma ), \mathcal{D}'(\mathfrak{P}^-_\Gamma )),\qquad  k_A \mapsto A,
  \end{gather*}
  where $ \langle A \varphi , \psi  \rangle \coloneqq \langle k_A , \varphi \otimes \psi \rangle.$
  We call $k_A$ the \emph{kernel} of $A$ and, for $f \in C^{\mathrm{l c}}(\mathfrak{P}_\Gamma)$, we write $f \cdot A $ for the operator with kernel $f \cdot k_A$.
\end{lem}

\begin{proof}
  Note that, since $ \mathfrak{P}^{\pm}_\Gamma$ are compact, we have the isomorphism
  \begin{equation*}
    C^{\mathrm{l c}} (\mathfrak{P}^+_\Gamma ) \otimes C^{\mathrm{l c}}(\mathfrak{P}^-_\Gamma ) \cong C^{\mathrm{l c}}(\mathfrak{P}^+_\Gamma  \times \mathfrak{P}^-_\Gamma ), \qquad f \otimes g \mapsto ((\mathbf{p}_+,\mathbf{p}_- ) \mapsto f(\mathbf{p}_+ )g(\mathbf{p}_-)).
  \end{equation*}
  Therefore, given $A \in \mathrm{Hom}(C^{\mathrm{l c}}(\mathfrak{P}^+_\Gamma ), \mathcal{D}'(\mathfrak{P}^-_\Gamma))$, the definition $\langle k_A, \varphi \otimes \psi \rangle \coloneqq \langle A \varphi , \psi \rangle$ determines a unique $k_A \in \mathcal{D}'(\mathfrak{P}^+_\Gamma \times \mathfrak{P}^-_\Gamma)$.
\end{proof}

\begin{definition}[Trace]
  For $A \in \mathrm{Hom}(C^{\mathrm{l c}}(\mathfrak{P}^+_\Gamma ), \mathcal{D}'(\mathfrak{P}^-_\Gamma ))$ we define its \emph{trace} as
  \begin{equation*}
    \mathrm{Tr}(A) \coloneqq \langle k_A , \mathbbm{1}_{\mathfrak{P}_\Gamma} \rangle,
  \end{equation*}
  where $k_A \in \mathcal{D}'(\mathfrak{P}^+_\Gamma \times \mathfrak{P}^-_\Gamma)$ denotes the kernel of $A$.
\end{definition}

\begin{rem}
  If $k_A$ is of the form $\sum_{
    i,j
  } a_{ij} u_i \otimes v_j$ for some $u_i \in \mathcal{D}'(\mathfrak{P}^+_\Gamma),\, v_j \in \mathcal{D}'(\mathfrak{P}^-_\Gamma )$ with $\langle u_i \otimes v_j, \mathbbm{1}_{\mathfrak{P}_\Gamma} \rangle = \delta _{ij}$ and $a_{ij} \in \mathbb{C}$ we have
  \begin{equation*}
    \mathrm{Tr}(A) = \sum_{i,j} a_{ij}\delta _{ij} = \sum_{i} a_{ii}.
  \end{equation*}
\end{rem}

\begin{definition}[Invariant Ruelle distribution]
  Let $s \in \mathbb{C}$ be a non-exceptional resonance and assume that for each $u\in \mathcal{E}_s(\mathcal{L}'_{\Gamma, +}; \mathcal{D}'(\mathfrak{P}_\Gamma^+))$
  \begin{equation} \label{eq:assump_Lq}
    (\mathcal{L}_{\Gamma, +}' - q^{\frac{1}{2} + is})^2 u = 0 \implies(\mathcal{L}_{\Gamma, +}' - q^{\frac{1}{2} + is})u = 0,
  \end{equation}
  i.e., that there are no Jordan blocks for $s$ (resp. ${-}\overline{s}$).
We define the \emph{invariant Ruelle distribution} $\mathcal{T}_s$ as continuous linear functional
  \begin{equation} \label{eq:invariantRuelledist}
    \mathcal{T}_s \colon
    \begin{cases}
        C^{\mathrm{l c}}(\mathfrak{P}_\Gamma) &\rightarrow \mathbb{C} \\
        \qquad f &\mapsto \mathrm{Tr}(f \Pi_s),
    \end{cases}
  \end{equation}
  where $\Pi_s= \sum_{\ell ,n=1}^m u_\ell \otimes v_{n}$ denotes the projection onto the eigenspace
  $$\mathcal{E}_s(\mathcal{L}'_{\Gamma , +}; \mathcal{D}'(\mathfrak{P}^+_\Gamma)) \otimes \mathcal{E}_{- \overline{s}}(\mathcal{L}'_{\Gamma , -};\mathcal{D}'(\mathfrak{P}^-_\Gamma)) \subseteq \mathcal{D}'(\mathfrak{P}_\Gamma^+ \times \mathfrak{P}^-_\Gamma).$$
\end{definition}

Note that $\Pi_s$ projects onto eigendistributions of the dual operator $\mathcal{L}'_\Gamma $ of the Ruelle transfer operator (see Definition~\ref{def:Ruelle_transfer_op}):
$\mathcal{L}'_\Gamma \Pi_s=q^{\frac{1}{2}+is} \Pi_s.$
Thus, the invariant Ruelle distribution satisfies the invariance property:
\begin{equation*} 
  \mathcal{L}'_\Gamma \mathcal{T}_s=q^{\frac{1}{2}+is}\mathcal{T}_s.
\end{equation*}

\begin{rem}
  In the Archimedean case, this functional is referred to as the \emph{invariant Ruelle distribution} \cite{GHWb}, since it is an invariant distribution associated with each Ruelle resonance. We have chosen to use the same name in the context of finite regular graphs, even though the terminology ``Ruelle resonances'' doesn't directly apply to this setting.
\end{rem}

\subsection{Relation to Patterson-Sullivan distributions}
\label{sect:relation_PSTs}
Using the pairing formula from \cite{AFH23Pairing}, the invariant Ruelle distributions are closely related to the Patterson–Sullivan distributions in the following way:

\begin{thm}[Relation with Patterson-Sullivan distributions]
\label{thm:relation_PSTs}
  Let $s \in \mathbb{C}$ be a resonance of multiplicity $m$ such that $q^{(\frac{1}{2} + is)} \notin \{\pm 1, \pm q\}$ and such that we do not have a Jordan block for $s$. Then $\chi(s)$ is an eigenvalue of $\Laplace_\Gamma $ and we have
  \begin{equation*}
      \mathcal{T}_s = \frac{q^{1 + 2 is} - 1}{q^{1 + 2 is} - q}\sum_{
        \ell = 1
      }^m \mathrm{P S}_{\phi_{\ell} , \phi_{\ell}}^\Gamma
    \end{equation*}
for any $\ell^2(\mathfrak{X}_\Gamma)$-orthonormal basis $\phi_1 , \ldots , \phi_m$ of $\mathcal{E}_{\chi(s)}(\Laplace_\Gamma; \mathrm{Maps}(\mathfrak{X}_\Gamma, \mathbb{C}))$.
  \end{thm}

  \begin{proof}
    Let
    \begin{equation*}
      u_\ell \coloneqq \frac{q^{1 + 2 is} - 1}{q^{1 + 2 is}- q}\mathbf{I}_+(s)(\phi_{\ell} ) \in \mathcal{E}_s(\mathcal{L}_{\Gamma , + }'; \mathcal{D}'(\mathfrak{P}^{+ }_\Gamma )), \quad v_{\ell} \coloneqq \mathbf{I}_-(- \overline{s})(\overline{\phi_{\ell}}) \in \mathcal{E}_{- \overline{s}}(\mathcal{L}_{\Gamma , - }'; \mathcal{D}'(\mathfrak{P}^{-}_\Gamma ))
    \end{equation*}
    which are defined and form bases of the respective spaces by the assumptions on $s$. Moreover, by Remark~\ref{rem:relation_pairing_formula} and Theorem~\ref{thm:PS_modern} the $v_{\ell} $'s are dual to the $u_{\ell}$'s with respect to the geodesic pairing, i.e., $(u_{\ell} \otimes v_{n})(\mathbbm{1}_{\mathfrak{P}_\Gamma}) = \delta_{\ell n}$. Thus, as there are no Jordan blocks, the projection $\Pi_s$ can be written as $\sum_{
      \ell,n
    } u_{\ell} \otimes v_{n}$ and we have for $f\in C^\mathrm{lc}(\mathfrak{P}_\Gamma)$
    \begin{equation*}
      \mathcal{T}_s(f) = \mathrm{Tr}(f \Pi_s) = \langle f\Pi_s , \mathbbm{1}_{\mathfrak{P}_\Gamma } \rangle= \sum_{
        \ell = 1
      }^m(u_{\ell} \otimes v_{\ell})(f) = \frac{q^{1 + 2 is} - 1}{q^{1 + 2 is} - q} \sum_{
        \ell = 1
      }^m \mathrm{P S}^\Gamma_{\phi_{\ell} , \phi_{\ell}}(f).\qedhere
    \end{equation*}
  \end{proof}

\section{Wigner distributions}
\label{sect:WignerDist}
The Wigner distributions associated with Laplacian eigenfunctions on finite $(q+1)$-regular graphs $\mathfrak{X}_\Gamma$, can, analogous to the Archimedean case \cite{AZ07, AZ12, HHS12, DP24}, be described on the phase space $S\mathfrak{X}_\Gamma$ via a choice of quantization map, namely a pseudo-differential operator calculus.
This calculus was conveniently adapted by Le Masson \cite{LeMas14} for regular graphs and to the quotient in \cite{ALe15}.

Let $a\colon \mathfrak{X} \times \Omega \times \C \rightarrow \C$ be a (bounded) measurable function. For any $(x,\omega, s) \in \mathfrak{X} \times \Omega \times \C$, we set
\begin{equation} \label{eq:pseudo_a}
    \mathrm{Op}(a) \, q^{(\frac{1}{2}+is)\langle x,\omega \rangle} \coloneqq q^{(\frac{1}{2}+is)\langle x,\omega \rangle} \, a(x, \omega, s).
\end{equation}
Now, for a general function with finite support, we define the pseudo-differential operator for regular graphs as follows.
\begin{definition}(Pseudo-differential operator for regular graphs)
\label{def:pseudo}
Let $a\colon \mathfrak{X} \times \Omega \times \C \rightarrow \C$ be a (bounded) measurable function. The pseudo-differential operator $\mathrm{Op}(a)$ is defined by
$$\mathrm{Op}(a)f(x)= \sum_{y \in \mathfrak{X}} \int_\Omega \int_{0}^{\frac{2 \pi}{\ln q}}
q^{(\frac{1}{2}-is)\langle y,\omega\rangle} q^{(\frac{1}{2}+is) \langle x,\omega\rangle} a(x, \omega, s) f(y) \; \mathrm{\dj{}}\xi(s)\mathrm{d}\nu_o(\omega) ,$$
for every $f\colon \mathfrak{X} \rightarrow \C$ with finite support.
Here $\mathrm{d}\nu_o(\omega)$ is the harmonic measure on $\Omega$ (see Remark~\ref{rem:proba_measure_nu}) and $\mathrm{\dj{}}\xi(s)$ is the Plancherel measure associated to the graph, which is given explicitly in \cite[Eq.~(2.8)]{LeMas14}.
\end{definition}

\begin{rem}[Conventions and complex spectral parameters] \label{rem:convention}
Compared to \cite{LeMas14}, we adopt a slightly different normalization of the
        Fourier--Helgason transform, namely the plane waves
        $q^{(\frac12-is)\langle x,\omega\rangle}$ instead of
        $q^{(\frac12+is)\langle x,\omega\rangle}$ as in \cite[§2.4]{LeMas14},
        which corresponds to using the adjoint
        Fourier--Helgason kernel and amounts to the harmless change of variable
        $s \mapsto -s$. This convention matches exactly the one used
        in the Archimedean setting of Zelditch's pseudodifferential calculus \cite{Zelditch86pseudo}.
        Moreover, in \cite[Def.~3.1]{LeMas14} the boundary measure is the harmonic measure
        $\mathrm{d}\nu_x(\omega)$, viewed from the base point $x\in\mathfrak X$ (see Remark~\ref{rem:proba_measure_nu}).
        Since $\mathrm{d} \nu_x(\omega) = q^{\langle x, \omega \rangle}\mathrm{d}\nu_o(\omega)$ we have, using \cite[(2.7)]{LeMas14},
        \begin{eqnarray*}
          \mathrm{Op}(a)f(x)
          &=& \int_{\Omega} \int_{0}^{\frac{2 \pi}{\ln q}} \mathrm{Op}(a) \, q^{(\frac{1}{2}+is)\langle x,\omega \rangle} \, \hat{f}(\omega,s) \; \mathrm{\dj{}}\xi(s)\, \mathrm{d}\nu_o(\omega)\\
          &=& \int_{\Omega} \int_{0}^{\frac{2 \pi}{\ln q}} q^{(\frac{1}{2}+is)\langle x,\omega \rangle} \, a(x, \omega, s)\, \hat{f}(\omega,s) \; \mathrm{\dj{}}\xi(s)\, \mathrm{d}\nu_o(\omega)\\
          &=& \sum_{y \in \mathfrak{X}} \int_{\Omega} \int_{0}^{\frac{2 \pi}{\ln q}} q^{(\frac{1}{2}-is)\langle y,\omega \rangle} q^{(\frac{1}{2}+is)\langle x,\omega \rangle} \, a(x, \omega, s)\, f(y) \; \mathrm{\dj{}}\xi(s)\, \mathrm{d}\nu_o(\omega),
        \end{eqnarray*}
        where $\hat{f}$ denotes the Fourier--Helgason transform of a finitely supported function $f: \mathfrak{X} \to \mathbb{C}$. Thus, Definition~\ref{def:pseudo} naturally extends Equation~\eqref{eq:pseudo_a}.
The involved choices are purely conventional and do not affect the pseudo-differential calculus. Note that, unlike the Fourier--Helgason transform, which depends on the choice of a reference point on the regular graph, the operator $\mathrm{Op}(a)$ is independent of this choice.
\end{rem}   

The pseudo-differential operator $\mathrm{Op}(a)$ on the quotient $\mathfrak{X}_\Gamma$ is equivalent to the $\Gamma$-invariant operator on $\mathfrak{X}$, since the elements of $\Gamma$ act without fixed points on $\mathfrak{X}$ \cite[§3.3]{ALe15}.
Moreover, it is not difficult to see that the operator $\mathrm{Op}(a)$ associated to a $\Gamma$-invariant bounded measurable function $a$ is $\Gamma$-invariant as well. 

In order to define Wigner distributions, we quantize \emph{locally constant functions}
$a \in C^{\mathrm{lc}}(S\mathfrak{X}_\Gamma)$ on the sphere bundle of $\mathfrak{X}_\Gamma$.
These functions play the role of \emph{tame symbols} in the microlocal sense. In the full pseudo-differential calculus developed by Le~Masson~\cite{LeMas14}, one works with
the symbol class $\mathcal{S}$ (or $\mathcal{S}_{\mathrm{sc}}$ in the semi-classical setting), which imposes regularity conditions on the boundary, a symmetry condition and control of the variations of $x\in \mathfrak{X}$.
Our locally constant functions form a dense subalgebra of $\mathcal{S}$ or $\mathcal{S}_{\mathrm{sc}}$
and suffice for the purpose of
defining Wigner distributions, although the finer properties such as the product and adjoint formulas  and the boundedness of $\mathrm{Op}(a)$ on $\ell^2(\mathfrak{X})$ require the  full symbol space, see \cite{LeMas14}.

\begin{rem}[Symbols on the quotient]
On the infinite regular tree $\mathfrak{X}$, the Laplacian has \emph{continuous spectrum}, parametrized by $s \in \C$ corresponding to the tempered spectrum. Consequently, the pseudo-differential operator $\mathrm{Op}(a)$ involves a symbol $a(x,\omega,s)$ depending on this continuous spectral parameter, integrated against the Plancherel measure $\mathrm{\dj{}}\xi(s)$.

When passing to a finite quotient $\mathfrak{X}_\Gamma = \Gamma \backslash \mathfrak{X}$, the Laplacian has \emph{discrete spectrum}, consisting of finitely many eigenvalues $\{\chi(s_j)\}^{N\in \N}_{j=1}$ and thus the spectral parameter $s$ is no longer continuous. As a result, for a $\Gamma$-invariant symbol $a$, the operator $\mathrm{Op}(a)$ on the quotient can be expressed using only the variables $(x,\omega) \in \mathfrak{X}_\Gamma \times \Omega$, with the dependence on $s$ entirely encoded in the spectral decomposition of the finite-dimensional operator $\mathrm{Op}(a)$. In other words, on $\mathfrak{X}_\Gamma$, the symbol $a$ can be naturally viewed as a function $a(x,\omega)$, without explicitly including $s$.
\end{rem}

Given $a\in C^{\mathrm{lc}}(S\mathfrak{X}_\Gamma)$, for an eigenfunction $\phi \in \mathcal{E}_{\chi(s)}(\Laplace_\Gamma; \mathrm{Maps}(\mathfrak{X}_\Gamma, \mathbb{C}))$ with spectral parameter $s\in \C$, we obtain by \eqref{eq:pseudo_a}
\begin{equation*}
    \mathrm{Op}(a)\phi(x)
= \int_\Omega q^{(\frac{1}{2}+is)\langle x,\omega\rangle} a(x,\omega) \; \mathrm{d}{\mu_{s,\phi}}(\omega), \quad (x\in \mathfrak{X}).
  \end{equation*}
The Wigner distribution is then defined in the following way:

  \begin{definition}(Wigner distribution)
    For $s,s'\in \C$, let $\phi \in \mathcal{E}_{\chi(s)}(\Laplace_\Gamma; \mathrm{Maps}(\mathfrak{X}_\Gamma, \mathbb{C})),\, \phi' \in \mathcal{E}_{\chi(s')}(\Laplace_\Gamma; \mathrm{Maps}(\mathfrak{X}_\Gamma, \mathbb{C}))$ be eigenfunctions. The \textit{Wigner distribution} $W_{\phi,\phi'} \in \mathcal{D}'(S\mathfrak{X}_\Gamma)$ associated to $\phi,\, \phi'$ is defined by
    $$W_{\phi,\phi'}(a)= \langle \mathrm{Op}(a) \phi, \phi' \rangle_{\ell^2(\mathfrak{X}_\Gamma)}, \quad (a\in C^{\mathrm{lc}}(S\mathfrak{X}_\Gamma)).$$
  \end{definition}

  Note that the pairing and thus the Wigner distribution is well-defined. Moreover, we can express the $\ell^2$-inner product via the Poisson transform defined in Definition~\ref{def:Poisson_transform},
  which yields the following formula for the Wigner distribution for $a\in C^{\mathrm{lc}}(S\mathfrak{X}_\Gamma)$:
  \begin{eqnarray} \label{eq:Wigner_dist}
    W_{\phi,\phi'}(a)
&=& \int_{\Omega \times \Omega} \Big(\sum_{x \in \mathfrak{X}} {\Xi (a \circ \pi_\Gamma)}(x,\omega) q^{(\frac{1}{2}+is)\langle x, \omega \rangle}q^{(\frac{1}{2}-i\overline{s}')\langle x, \omega' \rangle} \Big) \mathrm{d}\mu_{s,\phi}(\omega)\mathrm{d}\overline{\mu_{s',\phi'}}(\omega'),\nonumber \\
    &&
\end{eqnarray}
where $\Xi \in C^{\mathrm{lc}}_c(S\mathfrak{X})$ is a {lcfd-function} as in Definition~\ref{def:smoth_fundamental_domain_cutoff_fct}, $\pi_{\Gamma}: S\mathfrak{X} \rightarrow S\mathfrak{X}_\Gamma$ denotes the canonical projection, $\Xi  (a \circ \pi_\Gamma) \in C^{\mathrm{lc}}_c(S\mathfrak{X}),$
and $\mu_{s,\phi},\, {\mu_{s',\phi'}} \in \mathcal{D}'(\Omega)$ are the $\chi(s)$- and $\chi(s')$-boundary values, respectively, defined in \eqref{eq:boundary_values}.\\
The integral over $\Omega \times \Omega$ is understood in the sense of distributions, that is, as the action of the tensor product distribution $\mu_{s,\phi} \otimes \overline{\mu_{s',\phi'}}$. Moreover, the choice of the {lcfd-}function $\Xi$ in \eqref{eq:Wigner_dist} is irrelevant by Proposition~\ref{prop:cutoff_indep}.

\subsection{Relation to Patterson-Sullivan distributions}
\label{sect:relation_PSW}
The key to relating the two phase distributions is to express the Wigner distribution in terms of the Radon transform and an intertwining operator. To achieve this, one introduces a specific cut-off function that enables the oscillatory expansion of the Wigner distribution to be decomposed into two parts, each of which can be treated separately. This method originates from the Archimedean setting \cite{AZ07, AZ12, HHS12, DP24}.

For $a\in C^{\mathrm{lc}}(S\mathfrak{X}_\Gamma)$, consider the expansion $J_{s,s'}(a; \cdot, \cdot) \in C^{\mathrm{lc}}(\Omega \times \Omega)$, which is the inner sum over $\mathfrak{X}$ in the Wigner distribution description \eqref{eq:Wigner_dist}:
\begin{equation} \label{eq:J_ss'}
  J_{s , s'}(a;\omega , \omega ') \coloneqq \sum_{x \in \mathfrak{X}} {\Xi (a \circ \pi_\Gamma)}(x,\omega) q^{(\frac{1}{2}+is)\langle x, \omega \rangle}q^{(\frac{1}{2}-i\overline{s}')\langle x, \omega' \rangle} \quad \forall \omega,\omega' \in \Omega.
\end{equation}
For each $n \in \mathbb{N}_0$, we define the set (see Figure~\ref{fig:Sn})
\begin{equation} \label{eq:Sn}
  S_n \coloneqq \left\{(x, \omega , \omega ') \in \mathfrak{X} \times \Omega \times \Omega \mid d(x,]\omega , \omega '[) \leq n \right\},
\end{equation}
where $d(x,]\omega , \omega '[)$ is the distance from $x$ to the geodesic $]\omega , \omega '[= ( \dots, x_{-1},x_0,x_1, \dots)$, with $x_0=o$, and we denote by  $\mathbbm{1}_{S_n}$ the associated cutoff function so that
$$
    \mathbbm{1}_{S_n}(x, \omega, \omega') \coloneqq
    \begin{cases}
      1&\colon d(x,]\omega , \omega '[) \leq n, \\
      0&\colon \text{else.}
    \end{cases}
$$
Since each $S_n$ is $G$-invariant, the corresponding cutoff function $\mathbbm{1}_{S_n}$ and its complement $\mathbbm{1}_{S_n^C}$ are also $G$-invariant. Furthermore, both functions are locally constant \cite[Def.~4.1]{AFH23Pairing}.

\begin{figure}
  \centering

  \begin{tikzpicture}[scale=2, line width=0.8pt]
      \definecolor{myblue}{HTML}{0072B2}
      \definecolor{myorange}{HTML}{D55E00}
      \definecolor{mygreen}{HTML}{009E73}

      \newcommand{\ThreeRegTree}[3]{\ifnum#1>0\relax
          \pgfmathtruncatemacro{\nd}{#1-1}
          \pgfmathsetmacro{\nl}{#2*0.66}

          \ifnum#3=1\relax
            \foreach \ang in {90,210,-30}{
              \draw (0,0) -- ++(\ang:#2) coordinate (c);
              \fill (c) circle (0.9pt);
              \begin{scope}[shift={(c)}, rotate=\ang]
                \ThreeRegTree{\nd}{\nl}{0}
              \end{scope}
            }
          \else
            \foreach \ang in {60,-60}{
              \draw (0,0) -- ++(\ang:#2) coordinate (c);
              \fill (c) circle (0.7pt);
              \begin{scope}[shift={(c)}, rotate=\ang]
                \ThreeRegTree{\nd}{\nl}{0}
              \end{scope}
            }
          \fi
        \fi
      }

\newcommand{\TracePathSixSkipFirst}[8]{\path (0,0) -- ++(#2:#1) coordinate (c1);
        \begin{scope}[shift={(c1)}, rotate=#2]
          \pgfmathsetmacro{\lTwo}{#1*0.66*0.66}

\draw[#8, line width=2pt] (0,0) -- ++(#3:\lTwo) coordinate (c2);
          \begin{scope}[shift={(c2)}, rotate=#3]
            \pgfmathsetmacro{\lThree}{\lTwo*0.66}

\draw[#8, line width=2pt] (0,0) -- ++(#4:\lThree) coordinate (c3);
            \begin{scope}[shift={(c3)}, rotate=#4]
              \pgfmathsetmacro{\lFour}{\lThree*0.66}

\draw[#8, line width=2pt] (0,0) -- ++(#5:\lFour) coordinate (c4);
              \begin{scope}[shift={(c4)}, rotate=#5]
                \pgfmathsetmacro{\lFive}{\lFour*0.66}

\draw[#8, line width=2pt] (0,0) -- ++(#6:\lFive) coordinate (c5);
                \begin{scope}[shift={(c5)}, rotate=#6]
                  \pgfmathsetmacro{\lSix}{\lFive*0.66}

\draw[#8, line width=2pt] (0,0) -- ++(#7:\lSix);

                \end{scope}
              \end{scope}
            \end{scope}
          \end{scope}
        \end{scope}
      }

      \newcommand{\TracePathSixSkipFirstTwo}[8]{\path (0,0) -- ++(#2:#1) coordinate (c1);
        \begin{scope}[shift={(c1)}, rotate=#2]
          \pgfmathsetmacro{\lTwo}{#1*0.66*0.66}

\path (0,0) -- ++(#3:\lTwo) coordinate (c2);
          \begin{scope}[shift={(c2)}, rotate=#3]
            \pgfmathsetmacro{\lThree}{\lTwo*0.66}

\draw[#8, line width=2pt] (0,0) -- ++(#4:\lThree) coordinate (c3);
            \begin{scope}[shift={(c3)}, rotate=#4]
              \pgfmathsetmacro{\lFour}{\lThree*0.66}

\draw[#8, line width=2pt] (0,0) -- ++(#5:\lFour) coordinate (c4);
              \begin{scope}[shift={(c4)}, rotate=#5]
                \pgfmathsetmacro{\lFive}{\lFour*0.66}

\draw[#8, line width=2pt] (0,0) -- ++(#6:\lFive) coordinate (c5);
                \begin{scope}[shift={(c5)}, rotate=#6]
                  \pgfmathsetmacro{\lSix}{\lFive*0.66}

\draw[#8, line width=2pt] (0,0) -- ++(#7:\lSix);

                \end{scope}
              \end{scope}
            \end{scope}
          \end{scope}
        \end{scope}
      }

\node[circle,fill,inner sep=1.2pt,label={below:$x$}
      ] at (0,0) {};
      \draw[dashed, draw=gray!60] (0,0) circle (1.2);

      \ThreeRegTree{6}{1.2}{1}

      \TracePathSixSkipFirst{1.2}{210}{60}{-60}{60}{60}{-60}{myblue} \TracePathSixSkipFirst{1.2}{210}{-60}{-60}{60}{60}{60}{myblue}
      \TracePathSixSkipFirstTwo{1.2}{-30}{60}{60}{-60}{60}{-60}{myorange}
      \TracePathSixSkipFirstTwo{1.2}{-30}{60}{-60}{-60}{60}{60}{myorange}

    \end{tikzpicture}
\caption{Illustration of the set $S_1$: 
    In this case, $x$ together with the orange geodesic is \emph{not} part of $S_1$, since the distance $d(x, ]\omega , \omega '[)$ is $2$. On the other hand, $x$ together with the blue geodesic \emph{is} part of $S_1$, as here $d(x, ]\omega , \omega '[)$ is~1.}
\label{fig:Sn}
  \end{figure}

  We decompose
  \begin{equation*}
  J_{s,s'}(a;\omega , \omega ') = J_{s,s'}^{S_n}(a;\omega , \omega ') + J_{s,s'}^{S_n^C}(a; \omega , \omega '),
\end{equation*}
where
\begin{eqnarray}
    J_{s,s'}^{S_n}(a; \omega , \omega ') &\coloneqq& \sum_{x \in \mathfrak{X}} \mathbbm{1}_{S_n}(x, \omega ,\omega ') {\Xi (a \circ \pi_\Gamma)}(x,\omega) q^{(\frac{1}{2}+is)\langle x, \omega \rangle}q^{(\frac{1}{2}-i\overline{s}')\langle x, \omega' \rangle},  \label{eq:J_Sn}\\
  J_{s, s'}^{S_n^C}(a; \omega ,\omega ')                                    & \coloneqq& \sum_{x \in \mathfrak{X}} (1 - \mathbbm{1}_{S_n}(x, \omega ,\omega ')) {\Xi (a \circ \pi_\Gamma)}(x,\omega) q^{(\frac{1}{2}+is)\langle x, \omega \rangle}q^{(\frac{1}{2}-i\overline{s}')\langle x, \omega' \rangle}. \label{eq:J_SCn}
\end{eqnarray}

Note that the first summand \eqref{eq:J_Sn} describes the behaviour of $W_{\phi , \phi '}(a)$ supported off the diagonal in $\Omega^2$ (fixing $x \in \mathfrak{X}$, $d(x, ]\omega ,\omega '[) \leq n$ implies that $(\omega , \omega ') \in \Omega^2$ are not too close to each other), while the second one \eqref{eq:J_SCn} describes the behaviour close to the diagonal.

\subsubsection{The off-diagonal part} \label{sect:offdiagonal}
We first consider the summand \eqref{eq:J_Sn} for fixed $(\omega , \omega ')$. Since $(x, \omega , \omega ') \in S_n$ implies $\omega \neq \omega '$, Proposition~\ref{prop:open_cells} implies that there exists some $g \in G$ such that $(\omega ', \omega) =(g \omega_- , g \omega_+)$.
Summing over $gx$ instead of $x$ and using Equation~\eqref{eq:Haarmeasure_int}, we can rewrite $J_{s, s'}^{S_n}$ as
\begin{eqnarray*}
      J_{s, s'}^{S_n}(a;\omega , \omega ') 
      &=& \int_{B_{\omega_+}} \sum_{j \in \mathbb{Z}} q^{- j} \mathbbm{1}_{S_n}(g u \tau^j o, \omega , \omega '){\Xi (a \circ \pi_\Gamma)}(gu\tau^j o, \omega) \\
      && \hspace{7cm} q^{(\frac{1}{2} + is) \langle g u\tau^j o, \omega \rangle}q^{(\frac{1}{2} - i \overline{s}') \langle gu\tau^j o, \omega ' \rangle} \mathrm{d}u.
\end{eqnarray*}

By Lemma~\ref{lem:brackets_propI} we obtain
\begin{eqnarray*}
  \langle g u\tau^j o, \omega \rangle
    = \langle g u\tau^j o,g u \tau^j \omega_+ \rangle
    = H(gu \tau^j)
    &=& H(k(g) n(g) a(g)u \tau^j) \\
    &=& H(n(g) a(g) u a(g)^{-1} a(g) \tau^j) \\
    &=& H(a(g)\tau^j) = H(g \tau^j).
\end{eqnarray*}
Moreover, Lemmas~\ref{lem:brackets_propI}\ref{eq:horocycle_identity} and \ref{lem:brackets_propII} imply
\begin{equation*}
  \langle g u \tau^j o, \omega '\rangle  = \langle g u \tau^j o, g \tau^j \omega_- \rangle = \langle \tau^{- j}u \tau^j o, \omega_- \rangle + \langle g \tau^j o, g \tau^j \omega_- \rangle = - H(\tau^{-j} u^{-1} \tau^{j} r) + H(g \tau^j r).
\end{equation*}
Thus, with \eqref{eq:B_omega+}, we get
\begin{align*}
  J_{s, s'}^{S_n}(a;\omega , \omega ')
&= \sum_{j \in \mathbb{Z}}q^{-j} q^{(\frac{1}{2} + is)H(g \tau^j)} q^{(\frac{1}{2} - i \overline{s}')H(g \tau^j r)}\int_{\substack{B_{\omega_+}}}\mathbbm{1}_{S_n}(g u \tau^j o, g \omega_+, g \omega_-) \\
       &\hspace{4.2cm} \qquad \qquad {\Xi (a \circ \pi_\Gamma)}(g u \tau^j o, g \omega_+) q^{- (\frac{1}{2} - i \overline{s}')H(\tau^{-j} u^{-1} \tau^{j} r)}\mathrm{d}u\\
       &= \sum_{j \in \mathbb{Z}} q^{(\frac{1}{2} + is)H(g \tau^j)} q^{(\frac{1}{2} - i \overline{s}')H(g \tau^j r)}\int_{\substack{B_{\omega_+}}}\mathbbm{1}_{S_n}(g \tau^ju o, g \omega_+, g \omega_-) \\
       &\hspace{5cm} \qquad \qquad  {\Xi (a \circ \pi_\Gamma)}(g \tau^j u o, g \omega_+) q^{- (\frac{1}{2} - i \overline{s}')H(u^{-1} r)}\mathrm{d}u.
\end{align*}
We note that
\begin{equation*}
  \mathbbm{1}_{S_n}(g \tau^j u o, g \omega_+ , g \omega_-) = \mathbbm{1}_{S_n}(g \tau^j u o, g \tau^j \omega_+ , g \tau^j \omega_-) = \mathbbm{1}_{S_n}(u o, \omega_+ , \omega_-) =
  \begin{cases}
    1 &\colon \text{if } u \in B_{\omega_+ , n} \\
    0 &\colon \text{else},
  \end{cases}
\end{equation*}
where $B_{\omega_+ , n} \coloneqq \mathrm{Stab}_{B_{\omega_+}}(x_n)$. 
Arguing as in Proposition~\ref{prop:psi_hom}, we see that $S \mathfrak{X} \cong G/B_{\omega_+ ,0}$ so that, in group theoretic terms, $\Xi(a \circ \pi_\Gamma)$ defines a function on $G/B_{\omega_+ ,0}$:
\begin{equation*}
\Xi (a \circ \pi_\Gamma)(gB_{\omega_+ ,0}) \coloneqq \Xi(a \circ \pi_\Gamma)(go, g \omega_+).
  \end{equation*}
We note that, as $M \subset B_{\omega_+ ,0}$, this also defines a function on $G/M$. Using $g \omega_+ = g \tau^j u \omega_+$ and Lemma~\ref{lem:brackets_propII}, we may thus write
\begin{eqnarray*}
  J_{s, s'}^{S_n}(a;\omega , \omega ')
  &=& \sum_{
    j \in \mathbb{Z}
  } q^{(\frac{1}{2} + is)H(g \tau^j)} q^{(\frac{1}{2} - i \overline{s}')H(g \tau^j r)}\int_{
    \substack{B_{\omega_+,n}}
      } {\Xi (a \circ \pi_\Gamma)}(g \tau^j u) q^{\langle u o,\omega_- \rangle(\frac{1}{2} - i \overline{s}')}\mathrm{d}u\\
  &=&\sum_{
    j \in \mathbb{Z}
      }q^{(\frac{1}{2} + is)H(g \tau^j)} q^{(\frac{1}{2} - i \overline{s}')H(g \tau^j r)} (\mathcal{I}_{s',n}({\Xi (a \circ \pi_\Gamma)}))(g\tau^j) \\
  &=& \mathcal{R}_{s,s'}(\mathcal{I}_{s',n}({\Xi (a \circ \pi_\Gamma)}))(g),
\end{eqnarray*}
where
$\mathcal{R}_{s,s'}$ is the Radon transform defined in \eqref{eq:Radon_transform} and $\mathcal{I}_{s',n}$ defined by
\begin{equation} \label{eq:intw_op}
  \mathcal{I}_{s',n}: C^{\mathrm{lc}}_c(G/M) \rightarrow
  C_{c}^{\mathrm{lc}}(G/M), 
\qquad \mathcal{I}_{s',n}(f)(gM) \coloneqq \int_{
    B_{\omega_+, n}
  } f(guM) q^{\langle u o,\omega_- \rangle(\frac{1}{2} - i \overline{s}')}\mathrm{d}u
\end{equation}
approximates the \emph{intertwining operator} $\mathcal{I}_{s'} \coloneqq \lim_{n \rightarrow \infty} \mathcal{I}_{s',n}$. Note that $\mathcal{I}_{s', n}(f)$ is indeed right $M$-invariant: For $m \in M$ we have
\begin{equation*}
  \int_{
      B_{\omega_+, n}
  } f(gmuM) q^{\langle u o,\omega_- \rangle(\frac{1}{2} - i \overline{s}')}\mathrm{d}u =  \int_{
    B_{\omega_+, n}
  } f(guM) q^{\langle m^{-1} uo,\omega_- \rangle(\frac{1}{2} - i \overline{s}')}\mathrm{d}u,
\end{equation*}
and, by Lemma~\ref{lem:brackets_propI}\ref{eq:horocycle_identity},
\begin{equation*}
  \langle m^{-1} u o, \omega_- \rangle = \langle m^{-1} u o , m^{-1} \omega_- \rangle = \langle u o , \omega_- \rangle + \langle m^{-1} o, m^{-1} \omega_- \rangle = \langle u o, \omega_- \rangle.
  \end{equation*}
We explicitly compute the approximations $\mathcal{I}_{s',n}$ of the intertwining operator as follows:

\begin{lem} \label{lem:intw_op_expression}
  Consider $f \in C_c^{\mathrm{l c}}(S \mathfrak{X})$ as a function in $C_c^{\mathrm{l c}}(G/M)$. Then, for each $n \in \N_0$, we have 
  \begin{equation*}
        \mathcal{I}_{s', n}(f)(g) = \sum_{
        \substack{
          x \in H_{\omega_+}(o) \\
          d(x, o) \leq 2n 
        }
      }f(gx, g \omega_+)q^{- d(x, o)(\frac{1}{2} - i \overline{s}')}, \quad \forall g \in G,
    \end{equation*}
    where 
$H_{\omega_+}(o) = \left\{ x \in \mathfrak{X} \mid \langle x, \omega_+ \rangle = \langle o, \omega_+  \rangle = 0 \right\}$ 
    is the horocycle. Note that $\mathcal{I}_{s', 0}(f) = f$.
\end{lem}

\begin{proof}
  We prove the result by induction and claim furthermore that
  \begin{equation*}
       \int_{
        B_{\omega_+ , n}} f(g u) \mathrm{d} u = \sum_{
      \substack{
        x \in H_{\omega_+}(o) \\
        d(x, o) \leq 2n 
      }
    } f(gx, g \omega_+).
    \end{equation*}
  For $n = 0$, we have
  \begin{equation*}
    \int_{
      B_{\omega_+ , 0}
    } f(g u) q^{(\frac{1}{2} - i \overline{s}')\langle u o, \omega_- \rangle} \mathrm{d} u
    = \int_{
        B_{\omega_+ , 0}} f(g u) \mathrm{d} u 
    = f(g)\mathrm{vol}(B_{\omega_+ , 0}) = f(go, g \omega_+).
  \end{equation*}
  For the induction step, assume that the formulas hold for $n-1$ and choose representatives $u_1 , \ldots , u_q \in B_{\omega_+ , n}$ that map $x_{n - 1}$ to the neighbors of $x_n$ other than $x_{n + 1}$, with $u_1 = 1$. Then $B_{\omega_+ , n} = \bigsqcup_{i = 1}^q u_i B_{\omega_+ , n - 1} $ and we obtain
  \begin{equation*}
    \int_{
        B_{\omega_+ , n}} f(g u) \mathrm{d} u = \sum_{
      i = 1
    }^q \int_{B_{\omega_+ , n - 1}}f(g u_{i}u) \mathrm{d} u = \sum_{
      i = 1
    }^q \sum_{
      \substack{
        x \in H_{\omega_+}(o) \\
        d(x, o) \leq 2n - 2 
      }
    } f(gu_i x, g \omega_+) = \sum_{
      \substack{
        x \in H_{\omega_+}(o) \\
        d(x, o) \leq 2n 
      }
    } f(gx, g \omega_+).  
    \end{equation*}
Moreover, for $u \in \bigsqcup_{i = 2}^q u_i B_{\omega_+ , n - 1} = B_{\omega_+ , n} \setminus B_{\omega_+ , n - 1} $, we have
  \begin{equation*}
  \langle uo, \omega_- \rangle = d(o, o) - d(u o , o) = - d(u o , o) = - 2n.
\end{equation*}
Thus, $\int_{B_{\omega_+ , n}} f(gu) q^{\langle u o , \omega_- \rangle (\frac{1}{2} - i \overline{s}')} \mathrm{d} u$ equals
\begin{align*}
  &\hphantom{{}={}} q^{- 2n(\frac{1}{2} - i \overline{s}')}\sum_{i = 2}^q \int_{B_{\omega_+ , n - 1}} f(g u_i u) \mathrm{d} u + \int_{B_{\omega_+ , n - 1}} f(gu) q^{\langle u o, \omega_- \rangle (\frac{1}{2} - i \overline{s}')} \mathrm{d} u \\
  &= q^{- 2n(\frac{1}{2} - i \overline{s}')} \sum_{i = 2}^q \hspace{- 0.5 em}\sum_{\substack{x \in H_{\omega_+}(o) \\ d(x, o) \leq 2n - 2}} \hspace{- 1em}f(gu_i x, g \omega_+) +  \hspace{- 0.5 em}\sum_{\substack{x \in H_{\omega_+}(o) \\ d(x, o) \leq 2n - 2}}\hspace{- 1 em}f(gx, g \omega_+)q^{- d(x, o)(\frac{1}{2} - i \overline{s}')} \\
 &= q^{- 2n(\frac{1}{2} - i \overline{s}')} \sum_{\substack{x \in H_{\omega_+}(o) \\ d(x, o) = 2n}} f(g x, g \omega_+) +  \sum_{\substack{x \in H_{\omega_+}(o) \\ d(x, o) \leq 2n - 2}}\hspace{- 1 em}f(gx, g \omega_+)q^{- d(x, o)(\frac{1}{2} - i \overline{s}')}\\
  &= \sum_{\substack{x \in H_{\omega_+}(o) \\ d(x, o) \leq 2n}}f(gx, g \omega_+)q^{- d(x, o)(\frac{1}{2} - i \overline{s}')}.\qedhere
\end{align*}
\end{proof}

\begin{figure}
  \centering
\begin{tikzpicture}[scale=2, line width=0.8pt]
      \definecolor{myblue}{HTML}{0072B2}
      \definecolor{myorange}{HTML}{D55E00}
      \definecolor{mygreen}{HTML}{009E73}

      \newcommand{\TracePathThree}[5]{
        \draw[black, line width=0.8pt] (0,0) -- ++(#2:#1) coordinate (c1) node[circle, fill=black, inner sep=1.6pt] {};
        \begin{scope}[shift={(c1)}, rotate=#2]
          \pgfmathsetmacro{\lTwo}{#1*0.66*0.66}

\draw[black, line width=0.8pt] (0,0) -- ++(#3:\lTwo) coordinate (c2) node[circle, fill=black, inner sep=1.6pt] {};
          \begin{scope}[shift={(c2)}, rotate=#3]
            \pgfmathsetmacro{\lThree}{\lTwo*0.66}

\draw[black, line width=0.8pt] (0,0) -- ++(#4:\lThree) coordinate (c3) node[circle, draw=black, fill=#5, inner sep=1.6pt] {};
          \end{scope}
        \end{scope}
      }

      \newcommand{\DrawHorocycleLine}[1]{
\path (0,0) -- ++(320:#1) coordinate (c1);
        \begin{scope}[shift={(c1)}, rotate=320]
          \pgfmathsetmacro{\lTwo}{#1*0.66*0.66}
          \path (0,0) -- ++(0:\lTwo) coordinate (c2);
          \begin{scope}[shift={(c2)}, rotate=0]
            \pgfmathsetmacro{\lThree}{\lTwo*0.66}
            \path (0,0) -- ++(0:\lThree) coordinate (c3);
            \begin{scope}[shift={(c3)}, rotate=130]
              \draw[gray, dashed, line width=0.8pt] (-0.2,0) -- ++(0:3.1) node[pos=0, right] {$H_{\omega_+}(o)$};
              \node[gray, inner sep=0pt, xshift=3pt, yshift=38pt, right] at (0,0) {$o$};
            \end{scope}
          \end{scope}
        \end{scope}
      }
      \node[circle,fill,inner sep=1.6pt] at (0,0) {};
      \DrawHorocycleLine{1.2}
      \TracePathThree{1.2}{40}{0}{0}{mygreen}
      \TracePathThree{1.2}{40}{0}{280}{mygreen}
      \TracePathThree{1.2}{40}{280}{0}{mygreen}
      \TracePathThree{1.2}{40}{280}{80}{mygreen}
      \TracePathThree{1.2}{320}{0}{0}{myorange}
      \TracePathThree{1.2}{320}{0}{80}{myorange}
      \TracePathThree{1.2}{320}{80}{0}{myorange}
      \TracePathThree{1.2}{320}{80}{280}{myblue}
      \draw[black, line width=0.8pt]
      (0,0) -- ++(180:1.2) coordinate (c1)
      node[inner sep=0pt, xshift=-10pt, yshift=-0.3pt] {$\cdots$}
      node[gray, inner sep=0pt, xshift=-25pt, yshift=-1pt] {$\omega_+ $}
      node[circle, fill=black, inner sep=1.6pt] {};
    \end{tikzpicture}
    \caption{The gray line represents the horocycle $H_{\omega_+}(o)$. All the green vertices $y$ have $d(o,y)=6$, the orange vertices have $d(o,y)<6$.}
    \label{fig:horocycle}
\end{figure}

Recall from Definition~\ref{def:PS_GammaGM} that the Patterson–Sullivan distributions are defined as the $\Gamma$-average of the distributional Radon transform:
\begin{equation} \label{eq:PS_Radon}
    \mathrm{PS}^\Gamma_{\phi,\phi'}(a) = \mathcal{R}'_{s,s'}(\mu_{s,\phi} \otimes \overline{\mu_{s',\phi'}})(\Xi(a \circ \pi_\Gamma)) \quad a \in C^{\mathrm{lc}}(S\mathfrak{X}_\Gamma).
\end{equation}

Now, going all the way back to \eqref{eq:Wigner_dist} by integrating $J_{s,s'}^{S_n}$ over $\Omega \times \Omega$, we obtain that this off-diagonal part can be expressed in terms of the Radon transform and the intertwining operator for every $n \in \N_0.$

\begin{prop}[Off-diagonal part] \label{prop:off_diagonal_part}
  For $s,s' \in \C$, consider the eigenfunctions $\phi \in \mathcal{E}_{\chi(s)}(\Laplace; \mathrm{Maps}(\mathfrak{X}, \mathbb{C}))^\Gamma,\, \phi' \in \mathcal{E}_{\chi(s')}(\Laplace; \mathrm{Maps}(\mathfrak{X}, \mathbb{C}))^\Gamma$ and the corresponding boundary values $\mu_{s,\phi},\, \mu_{s',\phi'} \in \mathcal{D}'(\Omega)$.
  Let $S_n$ be the set as in \eqref{eq:Sn} for all $n\in \N$.
  Then we have for each $a\in C^{\mathrm{lc}}(S\mathfrak{X}_\Gamma)$:
  \begin{eqnarray*}
    W^{\mathrm{off}}_{\phi , \phi '}(a)
    &\coloneqq&
                \int_{\Omega \times \Omega} J^{S_n}_{s,s'}(a;\omega, \omega') \; \mathrm{d}\mu_{s,\phi}(\omega)\mathrm{d}\overline{\mu_{s',\phi'}}(\omega') \\
    &=& \mathcal{R}'_{s,s'}(\mu_{s,\phi} \otimes \overline{\mu_{s',\phi'}})(\mathcal{I}_{s',n}({\Xi (a \circ \pi_\Gamma)})),\quad \forall n \in \N_0,
  \end{eqnarray*}
  where $\mathcal{R}'_{s,s'}: \mathcal{D}'((\Omega \times \Omega)\backslash \mathrm{diag}(\Omega)) \rightarrow   \mathcal{D}'(G/M)$ is the distributional Radon transform defined in \eqref{eq:dualRadon}, $\mathcal{I}_{s',n}: C^{\mathrm{lc}}_c(G/M) \rightarrow C_c^{\mathrm{lc}}(G/M)$ is the intertwining operator defined in \eqref{eq:intw_op}, and $\Xi \in C^{\mathrm{lc}}_c(S\mathfrak{X})$ is an lcfd-function as in Definition~\ref{def:smoth_fundamental_domain_cutoff_fct}. Moreover,
  \begin{equation*}
    W^{\mathrm{off}}_{\phi , \phi '}(a) = \sum_{
      m = 0
    }^n q^{- 2m(\frac{1}{2} - i \overline{s}')} \mathrm{P S}^\Gamma_{\phi , \phi '}(\mathcal{H}_m(a)),
  \end{equation*}
  where $\mathcal{H}_m(a)(\Gamma g M) \coloneqq \mathcal{H}_m(a)(\Gamma(go, g\omega_+)) \coloneqq \mathcal{H}_m(a)(x_1 , x_2 , \ldots)$ is defined as the sum of $a(y_1 , \ldots, y_m , x_{m + 1}, x_{m + 2}, \ldots )$ over all non-backtracking paths with $y_m \neq x_m $. 
\end{prop}

\begin{proof}
  The first part follows from the above discussion. For the latter part, by Lemma~\ref{lem:intw_op_expression}, note that, for each $g \in G$ and $\gamma \in \Gamma$, we have
  \begin{align*}
    \sum_{
      \gamma \in \Gamma
    }\mathcal{I}_{s', n}(\Xi(a \circ \pi_\Gamma))(\gamma g) &= \sum_{
                                                              \gamma \in \Gamma
    } \sum_{
      \substack{
        x \in H_{\omega_+}(o) \\
        d(x, o) \leq 2n 
      }
    }\Xi (\gamma gx, \gamma g \omega_+) a(\Gamma(gx, g \omega_+))q^{- d(x, o)(\frac{1}{2} - i \overline{s}')}\\
                                                            &= \sum_{
                                                              \substack{
                                                              x \in H_{\omega_+}(o) \\
    d(x, o) \leq 2n 
    }
    } a(\Gamma(gx, g \omega_+))q^{- d(x, o)(\frac{1}{2} - i \overline{s}')}\\
                                          & = \sum_{
                                            m = 0
                                            }^n q^{- 2m(\frac{1}{2} - i \overline{s}')} \mathcal{H}_m(a)(\Gamma gM).\qedhere
  \end{align*}
\end{proof}

\subsubsection{The near-diagonal part} \label{sect:neardiagonal}
Fix $n\in \N$. We start to rewrite the second summand \eqref{eq:J_SCn} in terms of sums over paths. For this, we relate the set $S_n^C$ to paths by considering the set $\mathfrak{P}_n$ of all finite edge chains $\mathbf{p}=(\vec{e}_0, \dots, \vec{e}_n)$ of length $n+1$ in the universal cover $\mathfrak{G}.$ We denote by $p_0 \coloneqq \iota(\vec{e}_0), \dots, p_n \coloneqq \iota(\vec{e}_n)$ and $p_{n+1} \coloneqq \tau(\vec{e}_n)$ the vertices on $\mathbf{p} \in \mathfrak{P}_n$.
As in \cite[(16)]{AFH23Pairing}, we have
$$
(x, \omega, \omega') \in S^C_n \iff \exists ! \mathbf{p} \in \mathfrak{P}_n : p_0=x \text{ and } \omega, \omega' \in \partial_+(p_n,p_{n+1})
$$
so that $\mathbbm{1}_{S^C_n} = \sum_{\mathbf{p} \in \mathfrak{P}_n} \mathbbm{1}_{\mathbf{p}}$ with
$$
    \mathbbm{1}_{\mathbf{p}}(x, \omega, \omega') \coloneqq
    \begin{cases}
      1&\colon  p_0=x \text{ and } \omega, \omega' \in \partial_+(p_n,p_{n+1})\\
      0&\colon  \text{else.}
    \end{cases}
$$
Then, by using \cite[Lem.~4.12]{AFH23Pairing} for $a\in C^{\mathrm{lc}}(S\mathfrak{X}_\Gamma)$, we can rewrite the near-diagonal part of the Wigner distribution \eqref{eq:Wigner_dist} as
\begin{eqnarray}
    && W_{\phi,\phi'}^{\mathrm{near}}(a) \nonumber \\
    &\coloneqq& \int_{\Omega \times \Omega} J^{S^C_n}_{s,s'}(a;\omega, \omega') \; \mathrm{d}\mu_{s,\phi}(\omega)\mathrm{d}\overline{\mu_{s',\phi'}}(\omega') \nonumber \\
    &=& \int_{\Omega \times \Omega} \sum_{x \in \mathfrak{X}} \mathbbm{1}_{S^C_n}(x, \omega ,\omega ') {\Xi (a \circ \pi_\Gamma)}(x,\omega) q^{(\frac{1}{2}+is)\langle x, \omega \rangle}q^{(\frac{1}{2}-i\overline{s}')\langle x, \omega' \rangle} \; \mathrm{d}\mu_{s,\phi}(\omega)\mathrm{d}\overline{\mu_{s',\phi'}}(\omega') \nonumber \\
    &=& q^{-n(1+is-i\overline{s}')} \sum_{\mathbf{p} \in \mathfrak{P}_n} \int_{\partial_+(p_n,p_{n+1})} \int_{\partial_+(p_n,p_{n+1})} {\Xi (a \circ \pi_\Gamma)}(p_0,\omega) \nonumber \\
    &\qquad& \hspace{5.5cm} q^{(\frac{1}{2}+is)\langle p_n, \omega \rangle}q^{(\frac{1}{2}-i\overline{s}')\langle p_n, \omega' \rangle} \; \mathrm{d}\mu_{s,\phi}(\omega)\mathrm{d}\overline{\mu_{s',\phi'}}(\omega'). \label{eq:Iss'}
\end{eqnarray}

Consider now the set $\vec{\mathcal{E}}$ of all directed edges $\vec{e}=(y,z)$ pointing away from $x\in \mathfrak{X}$ such that $d(x,y)=n$. Then the expression \eqref{eq:Iss'} can be simplified as follows.
\begin{lem} \label{lem:near_diag_I_edges}
  Set, for $\vec{e} \in \vec{\mathcal{E}}$ and $\omega \in \partial_+ \vec{e}$:
  \begin{equation*}
    a_n (\iota(\vec{e}),\omega) \coloneqq \sum_{\substack{x \in \mathfrak{X} \\ d(x,\vec{e})=n}}
    a(x,\omega) \in C^{\mathrm{lc}}(S\mathfrak{X}_\Gamma),
  \end{equation*}
   where the sum runs over all vertices $x \in \mathfrak{X}$ such that there exists a path $(x,x_1, \dots, x_{n-1}, \vec{e})$, starting at $x$ and ending at the edge $\vec{e}$, of length $n$.
   Then we have
   \begin{equation} \label{eq:I_edgePoisson}
     W_{\phi,\phi'}^{\mathrm{near}}(a)
     = q^{-n(1+is-i\overline{s}')} \sum_{\vec{e} \in \vec{\mathcal{E}}} \int_{\partial_+\vec{e}}
     {(\Xi (a \circ \pi_\Gamma))_n}(\iota(\vec{e}),\omega)
     q^{(\frac{1}{2}+is)\langle \iota(\vec{e}), \omega \rangle} \; \mathrm{d}\mu_{s,\phi}(\omega)
     \mathcal{P}_s^e(\overline{\mu_{s',\phi'}}),
\end{equation}
where $\mathcal{P}_s^e(\overline{\mu_{s',\phi'}}) \coloneqq  \int_{\partial_+\vec{e}}  q^{(\frac{1}{2}-i\overline{s}')\langle \iota(\vec{e}), \omega' \rangle} \; \mathrm{d}\overline{\mu_{s',\phi'}}(\omega') \in \mathrm{Maps}(\mathfrak{E},\C)$ is the edge Poisson transform \cite[Def.~2.1]{AFH23} of $\overline{\mu_{s',\phi'}} \in \mathcal{D}'(\Omega)$.
\end{lem}

\begin{proof}
    We reorganize the sum over all paths $\mathbf{p} \in \mathfrak{P}_n$ in \eqref{eq:Iss'} by decomposing each path into its initial vertex $p_0=x$ and its edge $\vec{e}=(p_n, p_{n+1})$ with $d(x,p_n)=n$, which
    can be rewritten as the sum over all $x \in \mathfrak{X}$ and all directed edges $\vec{e} = (y,z)$ in $\vec{\mathcal{E}}$:
    \begin{align*}
      W_{\phi,\phi'}^{\mathrm{near}}(a)
&= q^{-n(1+is-i\overline{s}')} \sum_{x \in \mathfrak{X}} \sum_{\substack{\vec{e} \in \vec{
      \mathcal{E}} \\
      d(x, \vec{e}) = n
      }
      }\int_{\partial_+\vec{e}} \int_{\partial_+\vec{e}} {\Xi (a \circ \pi_\Gamma)}(x,\omega) \\
        &\hspace{7cm}  q^{(\frac{1}{2}+is)\langle y, \omega \rangle}q^{(\frac{1}{2}-i\overline{s}')\langle y, \omega' \rangle} \; \mathrm{d}\mu_{s,\phi}(\omega)\mathrm{d}\overline{\mu_{s',\phi'}}(\omega').
    \end{align*}
    By Fubini's theorem, we may interchange the order of summation over $x \in \mathfrak{X}$ and $\vec{e} \in \vec{\mathcal{E}}$, rewriting the sum so that for each edge $\vec{e}$ we sum over all vertices $x$ from which there exists a path of length $n$ ending at $\vec{e}$. This allows the integrals to depend only on $\vec{e}$:
    \begin{align*}
      W_{\phi,\phi'}^{\mathrm{near}}(a)
      &= q^{-n(1+is-i\overline{s}')} \sum_{\vec{e} \in \vec{\mathcal{E}}} \sum_{\substack{x \in \mathfrak{X} \\ d(x, \vec{e}) = n}}\int_{\partial_+\vec{e}} \int_{\partial_+\vec{e}} {\Xi (a \circ \pi_\Gamma)}(x,\omega) \\
      &\qquad \hspace{5cm} q^{(\frac{1}{2}+is)\langle \iota(\vec{e}), \omega \rangle}q^{(\frac{1}{2}-i\overline{s}')\langle \iota(\vec{e}), \omega' \rangle} \; \mathrm{d}\mu_{s,\phi}(\omega)\mathrm{d}\overline{\mu_{s',\phi'}}(\omega')\\
      &= q^{-n(1+is-i\overline{s}')} \sum_{\vec{e} \in \vec{\mathcal{E}}} \int_{\partial_+\vec{e}}
        {(\Xi (a \circ \pi_\Gamma))_n }(\iota(\vec{e}),\omega)
        q^{(\frac{1}{2}+is)\langle \iota(\vec{e}), \omega \rangle} \; \mathrm{d}\mu_{s,\phi}(\omega)\\
      &\qquad \hspace{7.7cm} \int_{\partial_+\vec{e}}  q^{(\frac{1}{2}-i\overline{s}')\langle \iota(\vec{e}), \omega' \rangle} \; \mathrm{d}\overline{\mu_{s',\phi'}}(\omega').
    \end{align*}
  \end{proof}
  Now observe that we can rewrite $\partial_+ \vec{e}$ as $\Omega \backslash \vec{\mathcal{B}}$, where 
  $$\vec{\mathcal{B}} \coloneqq \{\partial_{+}\vec{b} \mid
  \vec{b}\in\vec{\mathcal E}:
  \iota(\vec{b})=\iota(\vec{e}) \text{ and }
  \vec{b}\neq \vec{e}\}.$$
  Hence the edge Poisson transform appearing in \eqref{eq:I_edgePoisson} can be rewritten as
  \begin{eqnarray} \label{eq:expression_edgePoisson}
  \mathcal{P}_s^e(\overline{\mu_{s',\phi'}})
&=& \int_\Omega q^{(\frac{1}{2}-i\overline{s}')\langle \iota(\vec{e}), \omega' \rangle} \; \mathrm{d}\overline{\mu_{s',\phi'}}(\omega')
      - \sum_{\vec{b} \in \vec{\mathcal{B}}} \int_{\partial_{+}\vec{b}} q^{(\frac{1}{2}-i\overline{s}')\langle \iota(\vec{e}), \omega' \rangle} \; \mathrm{d}\overline{\mu_{s',\phi'}}(\omega').
\end{eqnarray}

Combining everything together we obtain the following expression:

\begin{prop}[Near-diagonal part] \label{prop:near_diagonal_part}
  Consider the same parameters as in Proposition~\ref{prop:off_diagonal_part}. For each $x\in \mathfrak{X}$ and $\omega \in \Omega$, let $x_\omega \coloneqq x_1$ be the neighbor of $x$ in the direction of $\omega$, i.e.,\@ the first vertex along the geodesic $[x,\omega)=(x=x_0,x_1, x_2, \dots)$.
  For
  \begin{equation} \label{eq:a_tilde}
    a_n(x,\omega)
    = \sum_{\substack{y \in \mathfrak X \\ d(y,{(x, x_\omega)}) = n}} a(y,\omega) \in {C^{\mathrm{lc}}(S\mathfrak{X}_\Gamma)},
  \end{equation} 
  we then have 
  $$W_{\phi,\phi'}^{\mathrm{near}}(a) = q^{-n(1+is-i\overline{s}')} \Big(W_{\phi,\phi'}(a_n) -\mathrm{PS}^\Gamma_{\phi,\phi'}(a_n ) \Big).$$
\end{prop}

\begin{proof}
  Adding \eqref{eq:expression_edgePoisson} and \eqref{eq:I_edgePoisson} yields
  $$
  W_{\phi,\phi'}^{\mathrm{near}}(a) = q^{-n(1+is-i\overline{s}')}(Y_1(a) - Y_2(a)),
  $$
    where
    \begin{align*}
      Y_1(a)
      &\coloneqq
        \sum_{\vec{e} \in \vec{\mathcal{E}}} \int_{\partial_+\vec{e}}
        {(\Xi (a \circ \pi_\Gamma))_n }(\iota(\vec{e}),\omega)
        q^{(\frac{1}{2}+is)\langle \iota(\vec{e}), \omega \rangle} \; \mathrm{d}\mu_{s,\phi}(\omega)
        \int_\Omega q^{(\frac{1}{2}-i\overline{s}')\langle \iota(\vec{e}), \omega' \rangle} \; \mathrm{d}\overline{\mu_{s',\phi'}}(\omega') \\
      Y_2(a)
      &\coloneqq
        \sum_{\vec{e} \in \vec{\mathcal{E}}} \int_{\partial_+\vec{e}}
        {(\Xi (a \circ \pi_\Gamma))_n}(\iota(\vec{e}),\omega)
        q^{(\frac{1}{2}+is)\langle \iota(\vec{e}), \omega \rangle} \; \mathrm{d}\mu_{s,\phi}(\omega)
        \sum_{\vec{b} \in \vec{\mathcal{B}}} \int_{\partial_{+}\vec{b}} q^{(\frac{1}{2}-i\overline{s}')\langle \iota(\vec{e}), \omega' \rangle} \; \mathrm{d}\overline{\mu_{s',\phi'}}(\omega').
    \end{align*}
    First note that the sum over all directed edges $\vec{e} \in \vec{\mathcal{E}}$ together with the integral over $\partial_+\vec{e}$ can be rewritten as a double sum over vertices $x \in \mathfrak{X}$ and their neighbors $y \in \mathfrak{X}$ at distance one. This is possible because each directed edge $\vec{e}=(x,y)$ is uniquely determined by its initial vertex $x$ and its terminal vertex $y$ with $d(x,y)=1$. Moreover, for each $x \in \mathfrak{X}$,
    \[
    \bigcup_{\substack{y \in \mathfrak{X} \\ d(x,y)=1}} \partial_+(x,y) = \Omega,
    \]
    so that the sum over all $x \in \mathfrak{X}$ with an integral over $\partial_+(x,y) \coloneqq \{ \omega \in \Omega \;|\; [x, \omega)$ \text{ starts with the edge} $(x,y)\}$ for each neighbor $y$ of $x$ can be replaced by a sum over $x \in \mathfrak{X}$ with an integral over $\Omega$.
    Hence, the first integral simplifies as follows:
    \begin{eqnarray*}
      Y_1(a)
      &=& \sum_{x \in \mathfrak{X}} \sum_{\substack{y \in \mathfrak{X} \\ d(x,y)=1}} \int_{\partial_+(x,y)} {(\Xi (a \circ \pi_\Gamma))_n }(x,\omega) q^{(\frac{1}{2}+is)\langle x, \omega \rangle} \; \mathrm{d}\mu_{s,\phi}(\omega)
      \int_\Omega q^{(\frac{1}{2}-i\overline{s}')\langle x, \omega' \rangle} \; \mathrm{d}\overline{\mu_{s',\phi'}}(\omega')\\
      &=& \sum_{x \in \mathfrak{X}} \int_{\Omega} {(\Xi (a \circ \pi_\Gamma))_n }(x,\omega) q^{(\frac{1}{2}+is)\langle x, \omega \rangle} \; \mathrm{d}\mu_{s,\phi}(\omega)
          \int_\Omega q^{(\frac{1}{2}-i\overline{s}')\langle x, \omega' \rangle} \; \mathrm{d}\overline{\mu_{s',\phi'}}(\omega') \\
      &=& \sum_{x \in \mathfrak{X}} \int_{\Omega \times \Omega} {(\Xi (a \circ \pi_\Gamma))_n }(x,\omega) q^{(\frac{1}{2}+is)\langle x, \omega \rangle} q^{(\frac{1}{2}-i\overline{s}')\langle x, \omega' \rangle} \; \; \mathrm{d}\mu_{s,\phi}(\omega) \mathrm{d}\overline{\mu_{s',\phi'}}(\omega') \\
      &=& W_{\phi,\phi'}(a_n),
    \end{eqnarray*}
    where the last equality follows from
    \begin{align*}
      \sum_{\gamma \in \Gamma} (\Xi (a \circ \pi_\Gamma))_n(\gamma x, \gamma \omega) 
      &= \sum_{\gamma \in \Gamma}  \sum_{\substack{y \in \mathfrak X \\ d(y,{(\gamma x, x_{\gamma \omega })}) = n}} \Xi (a \circ \pi_\Gamma)(y,\gamma \omega )\\
       & =    \sum_{\substack{y \in \mathfrak X \\ d(y,{(x, x_{\omega })}) = n}} a(\Gamma (y,\omega ))\sum_{
      \gamma \in \Gamma
      }\Xi(\gamma y,\gamma \omega ) = a_n(\Gamma(x, \omega)).
      \end{align*}
      As for the second integral expression, note that $\mathfrak{P} \subseteq \mathfrak{X} \times \Omega \times \Omega$, hence for $(x,\omega, \omega') \in \mathfrak{X} \times \Omega \times \Omega$ such that $x\in ]\omega, \omega'[$ we have
      $$\mathds{1}_{\mathfrak{P}}(x,\omega, \omega') = \mathds{1}_{S_0}(x,\omega, \omega').$$
      Therefore
      \begin{eqnarray*}
        Y_2(a_n)
        &=& \sum_{x \in \mathfrak{X}} \int_{\Omega \times \Omega} \mathds{1}_{S_0}(x,\omega, \omega') {(\Xi (a \circ \pi_\Gamma))_n }(x,\omega) q^{(\frac{1}{2}+is)\langle x, \omega \rangle} q^{(\frac{1}{2}-i\overline{s}')\langle x, \omega' \rangle} \; \mathrm{d}\mu_{s,\phi}(\omega) \mathrm{d}\overline{\mu_{s',\phi'}}(\omega') \\
        &=& \int_{\Omega \times \Omega} J^{S_0}_{s,s'}(a_n ; \omega, \omega') \; \mathrm{d}\mu_{s,\phi}(\omega) \mathrm{d}\overline{\mu_{s',\phi'}}(\omega')\\
        &=& \mathrm{PS}^\Gamma_{\phi,\phi'}(a_n ),
      \end{eqnarray*}
      where the last line follows from Proposition~\ref{prop:off_diagonal_part} and \eqref{eq:PS_Radon}.
    \end{proof}

    \subsubsection{Relation} \label{sect:relation}
    Before relating the Wigner distribution to the Patterson–Sullivan distribution, we first relate the $ a_n\in C^{\mathrm{lc}}(S\mathfrak{X}_\Gamma)$, defined in \eqref{eq:a_tilde}, with the (Ruelle) transfer operator $\mathcal{L}$ in Definition~\ref{def:Ruelle_transfer_op}.
Let $\vec e = (x, x_\omega)$ denote the forward edge from $x$ in the direction of $\omega \in \Omega$. Then $a_n (x,\omega)$ can be interpreted as follows:

    \begin{itemize}
    \item \textbf{Case $n=0$:} The only vertex at distance $0$ from $\vec e$ is the initial vertex $x$, i.e., there is no shift along the graph. Therefore, the sum has only one term $y=x$, and
    $$
      a_0 (x,\omega) = a(x,\omega).
    $$

    \item \textbf{Case $n=1$:} This corresponds to summing over the \emph{other neighbors} of $x$ (excluding the forward neighbor $x_\omega$ in the direction of $\omega$). Equivalently, this is the action of the transfer operator $\mathcal{L}$ applied to $a$ on  $\mathfrak{E}_\Gamma \coloneqq  \Gamma \backslash \mathfrak{E}$
      , pulled back to $S\mathfrak X_\Gamma$ via $(x,\omega) \mapsto (x,x_\omega)$:
      \begin{equation*}
        a_1 (x,\omega) 
        = \sum_{y: \sigma_\omega(y)=x} a(y,\omega) =  
\mathcal{L}a(x,\omega),
        \end{equation*}

      where $\sigma_\omega(y)=\sigma(\omega,y)=x$ is the shift map along $\omega$.
    \item \textbf{Case $n>1$:} This corresponds to summing over all vertices at distance $n$ from the forward edge $\vec e$. Equivalently, this is obtained by applying $\mathcal{L}$ $n$-times:
      $$
      a_n (x,\omega) = \mathcal{L}^n a(x,\omega).
    $$
\end{itemize}

Recalling formula~\eqref{eq:Wigner_dist} for the Wigner distribution:
$$ W_{\phi,\phi'}(a)=  W_{\phi,\phi'}^{\mathrm{off}}(a) + W_{\phi,\phi'}^{\mathrm{near}}(a), \quad \text{ for } a\in C^{\mathrm{lc}}(S\mathfrak{X}_\Gamma),$$
and combining the off-diagonal and near-diagonal expressions from Propositions~\ref{prop:off_diagonal_part} and~\ref{prop:near_diagonal_part}, respectively, we observe that the resulting expression for the Wigner distribution is related to the Patterson–Sullivan distributions via the intertwining operator $\mathcal{I}_{s',n}$. Using the explicit formula for $\mathcal{I}_{s',n}$ in Lemma~\ref{lem:intw_op_expression}, we thus proved the following result:

\begin{thm}[Relation between Patterson-Sullivan and Wigner distributions]
    \label{thm:relation_PSW}
    For $s,s' \in \C$, consider the eigenfunctions $\phi \in \mathcal{E}_{\chi(s)}(\Laplace_\Gamma; \mathrm{Maps}(\mathfrak{X}_\Gamma, \mathbb{C})), \phi' \in \mathcal{E}_{\chi(s')}(\Laplace_\Gamma; \mathrm{Maps}(\mathfrak{X}_\Gamma, \mathbb{C}))$ and the corresponding boundary values $\mu_{s,\phi}, \mu_{s',\phi'} \in \mathcal{D}'(\Omega)$.
    Then we have for each $a\in C^{\mathrm{lc}}(S\mathfrak{X}_\Gamma)$ and $n\in \N_0$:
    \begin{align*}
      W_{\phi,\phi'}\Big(a-q^{-n(1+is-i\overline{s}')}\mathcal{L}^n{a}\Big) 
      &= \sum_{m=0}^{n} q^{-2m(\frac12 - i\overline{s}')} \,\mathrm{PS}^\Gamma_{\phi,\phi'}\Big(\mathcal{H}_m(a)\Big)\\
      &\qquad \qquad -q^{-n(1+is-i\overline{s}')} \mathrm{PS}^\Gamma_{\phi,\phi'}\Big(\mathcal{L}^n a \Big)\\
      &= \mathrm{PS}^\Gamma_{\phi,\phi'} \Big(\sum_{m=0}^{n} q^{-2m(\frac12 - i\overline{s}')}\mathcal{H}_m(a) - q^{-n(1+is-i\overline{s}')}\mathcal{L}^n a \Big),
    \end{align*}
    where
    $\mathcal{H}_m$ is defined as in Proposition~\ref{prop:off_diagonal_part} and $\mathcal{L}$ denotes the (Ruelle) transfer operator in Definition~\ref{def:Ruelle_transfer_op}.
\end{thm}
  Note that for $n=0$, the Wigner and Patterson–Sullivan distributions cancel out. 
    Moreover, Theorem~\ref{thm:relation_PSW} reads
  \begin{equation*}
    W_{\phi , \phi '}(a - q^{ni(\overline{s}' - s)}q^{- n}\mathcal{L}^n a) = \mathrm{P S}^\Gamma_{\phi , \phi '}\left(a - q^{ni(  \overline{s}' - s)}q^{- n}\mathcal{L}^n a + \sum_{
        m = 1
      }^n q^{2mi \overline{s}'}q^{- m}\mathcal{H}_m(a)\right).
  \end{equation*}
  Here, the factor $q^{- n}$ (resp.\@ $q^{- m}$) can be viewed as a scaling factor, since $\mathcal{L}^n$ (resp.\@ $\mathcal{H}_m$) consists of $q^n$ (resp.\@ $q^{m - 1}(q - 1)$) summands. More importantly, if $- \mathrm{Im}(s + s')$ is large, $\lvert q^{ni(\overline{s}' - s)} \rvert = q^{n \mathrm{Im}(s + s')}$ and $\lvert q^{2mi \overline{s}'} \rvert = q^{2m \mathrm{Im}(s')}$ become small so that $W_{\phi , \phi '}(a)$ gets close to $\mathrm{P S}^\Gamma_{\phi , \phi '}(a)$. This can be seen as an analogue of the Archimedean case, although we cannot consider a limit for $\mathrm{Im}(s + s') \rightarrow - \infty$ in our setting since we only have finitely many resonances.

  \begin{ex}[Basic example]
      Consider the constant function $a = \mathbbm{1} \in C^{\mathrm{lc}}(S\mathfrak{X}_\Gamma)$. 
      For all $n\in \N_0$, we have $\mathcal{L}^n a = a_n = q^n\mathbbm{1}$ and $\mathcal{H}_m(a) = q^{m - 1}(q - 1)\mathbbm{1}$ for $m \geq 1$. 
      Thus, Theorem~\ref{thm:relation_PSW} for $n = 1$ yields 
      \begin{align*}
        W_{\phi , \phi '}(\mathbbm{1}(1 - q^{i(\overline{s}'-s)})) 
    &= \mathrm{P S}^\Gamma_{\phi , \phi '}(\mathbbm{1}(1 - q^{i(\overline{s}'-s)} + q^{2i\overline{s}'-1}(q - 1) ))\\
    &=\mathrm{P S}^\Gamma_{\phi , \phi '}\left(\mathbbm{1}\left(1-q^{i(\overline{s}'-s)}+\frac{q^{i\overline{s}'}}{q^{-i\overline{s}'}}(1-q^{-1})\right)\right).
      \end{align*}
By normalizing both expressions by the factor $(1-q^{i(\overline{s}'-s)})^{-1}$, we get
      \begin{equation*}
        W_{\phi , \phi '}(\mathbbm{1}) 
        = \mathrm{P S}^\Gamma_{\phi , \phi '}\left(\mathbbm{1} \left(1 + \frac{q^{i \overline{s}'} - q^{i \overline{s}' - 1}}{q^{- i \overline{s}'} - q^{- is}}\right)\right).
      \end{equation*}
  In the case $s = - \overline{s}'$, this becomes
  \begin{equation*}
    W_{\phi , \phi '}(\mathbbm{1}) = \mathrm{P S}^\Gamma_{\phi , \phi '}\left(\frac{\mathbbm{1}}{\sqrt{q}}\frac{q^{\frac{1}{2} + is} - q^{- \frac{1}{2} - is}}{q^{is} - q^{- is}}\right) = \mathrm{P S}^\Gamma_{\phi , \phi '}\left(\mathbf{c}(s)(1+q^{-1})\mathbbm{1}\right),
  \end{equation*}
  where $\mathbf{c}(s)\coloneqq\frac{q^{1/2}}{q+1}\frac{q^{\frac{1}{2} + is} - q^{- \frac{1}{2} - is}}{q^{is} - q^{- is}}$ is the $\mathbf{c}$-function (see e.g.\@ \cite[(2.9)]{LeMas14}).
\end{ex}

\bibliographystyle{amsalpha}
\bibliography{Literatur.bib}
\end{document}